\documentclass[10pt,leqno]{amsart}
\usepackage{graphicx}
\usepackage{indentfirst,csquotes}
\usepackage{geometry}
\usepackage{amssymb,amsthm,amsmath}
\usepackage{xcolor,paralist,fancyhdr,etoolbox}
\usepackage[colorlinks,linkcolor=blue,citecolor=blue,hypertexnames=false]{hyperref}

\usepackage{mathtools}

\usepackage{tikz,placeins}
\newcommand{\supp}[1]{\text{\normalfont supp} {(#1)}}

\newcommand{\diam}{\text{\normalfont diam}}
\newcommand{\ldimqhat}[2]{\underline{\widehat{\dim}}_Q^{#1}{(#2)}}

\newcommand{\udimqhat}[2]{\overline{\widehat{\dim}}_Q^{#1}{(#2)}}
\newcommand{\ldimqtilde}[2]{\underline{\widetilde{\dim}}_Q^{#1}{(#2)}}
\newcommand{\udimqtilde}[2]{\overline{\widetilde{\dim}}_Q^{#1}{(#2)}}

\newtheorem{theorem}{Theorem}

\newtheorem{example}[theorem]{Example}
\newtheorem{lemma}[theorem]{Lemma}
\newtheorem{proposition}[theorem]{Proposition}
\newtheorem{corollary}[theorem]{Corollary}

\newtheorem{remark}{Remark}

\hypersetup{ colorlinks=true, linkcolor=black, filecolor=black, urlcolor=black,
pdftitle={Asymptotics of Constrained Quantization for Compactly Supported Measures},
pdfauthor={Chenxing Qian} }

\usepackage[T1]{fontenc}
\usepackage{lmodern,microtype,booktabs,tabularx,longtable,array}
\allowdisplaybreaks[2]
\newcommand{\R}{\mathbb R}
\newcommand{\N}{\mathbb N}
\newcommand{\dist}{\operatorname{dist}}

\newcommand{\Acal}{\mathcal A}
\newcommand{\HD}{\dim_H^*}
\newcommand{\Blo}{\underline{\dim}_B}
\newcommand{\Bup}{\overline{\dim}_B}
\newcommand{\Dlo}{\underline D}
\newcommand{\Dup}{\overline D}
\newcommand{\one}{\mathbf 1}
\numberwithin{equation}{section}

\usepackage{esint}
\newcommand{\Vcal}{\mathcal V}
\newcommand{\Hcal}{\mathcal H_{\mathrm n}}
\newcommand{\wh}{\widehat w}

\begin{document}
\title[Asymptotics of Constrained Quantization]{Asymptotics of Constrained Quantization for Compactly Supported Measures}
\author[Qian]{Chenxing Qian}
\address{Department of Mathematics, University of Pittsburgh,
Room 605, Thackeray Hall, 166 Thackeray Ave,
Pittsburgh, PA 15260, USA}
\email{qianchenxing@outlook.com}
\date{}
\begin{abstract}
We study constrained quantization of a compactly supported probability $P$
on $\R^D$ using codebooks in a nonempty closed set $S$. For $0<r\le\infty$,
the constrained quantization error and its baseline are
\[
 e_{n,r}=\inf_{\substack{\alpha\subset S\\1\le|\alpha|\le n}}
 \|\dist(\cdot,\alpha)\|_{L^r(P)},\qquad
 L_r=\|\dist(\cdot,S)\|_{L^r(P)}.
\]
We establish attainment and convergence $e_{n,r}\downarrow L_r$.
For finite $r$, set $\Delta_{n,r}=e_{n,r}^r-L_r^r$.
A Borel nearest-point selector $T$ gives a projected measure $\mu=T_*P$.
For a probability $\eta$ and $s>0$, define the quantization error
\[
 q_{n,s}(\eta)=\inf_{\substack{\alpha\subset\R^D\\1\le|\alpha|\le n}}
 \|\dist(\cdot,\alpha)\|_{L^s(\eta)}.
\]
We obtain
\[
 e_{n,r}-L_r\le2q_{n,r}(\mu),\qquad
 \Delta_{n,r}\le C_rq_{n,1}(\mu)\quad(1\le r<\infty),
\]
with $C_r$ independent of $n$, and
$\Delta_{n,r}\le2^rq_{n,r}(\mu)^r$ for $0<r<1$.
In addition, $e_{n,\infty}-L_\infty\le2c_n(E)$, where $E$ is the
nearest-point image of $\supp P$ and $c_n(E)$ is its covering radius.
For a finite measure $\nu$, define the quantization cost
\[
 \Vcal_{n,\beta}(\nu)=\inf_{\substack{\alpha\subset\R^D\\1\le|\alpha|\le n}}
 \int\dist(y,\alpha)^\beta\,d\nu(y).
\]
For $1\le r<\infty$, a weighted power-growth condition with exponent
$\beta>0$ on a fixed source set $F$ gives
$\Delta_{n,r}\ge\Vcal_{n,\beta}(\nu_F)$, where
$\nu_F=T_*(wP|_F)$. If $d\nu_F=\wh\,d\eta$ for a compactly supported
reference probability $\eta$ and $\wh\in A_p(\eta)$, $1<p<\infty$, then
\[
 \Delta_{n,r}\ge\frac{\nu_F(\R^D)}{[\wh]_{A_p(\eta)}}
 q_{n,\beta/p}(\eta)^\beta.
\]
These comparisons yield bounds for the quantization dimensions defined by
$e_{n,r}-L_r$ and $\Delta_{n,r}^{1/r}$, with distinct behavior at positive
and zero baselines. Convex and uniformly prox-regular constraints supply
the growth weights through their projection inequalities.
Circle and lifted Cantor distributions give matching quadratic-excess
rates, while rectifiable weighted projections give lower Zador estimates.
\end{abstract}
\maketitle

\section{Introduction}\label{section:introduction}
\subsection{Background and the problem}
Quantization studies the approximation of a probability distribution by a
finite set of representatives. Its asymptotic theory relates the number of
representatives and the approximation error to the distribution of mass;
see Graf and Luschgy~\cite{graf} for the general theory and
Gray and Neuhoff~\cite{gray} for its information-theoretic background.
For singular probabilities, the quantization dimension also records features
of the measure that enter the approximation rate; self-similar distributions
provide a central class of examples~\cite{graf2000asymptotics,graf2002quantization}.

In constrained quantization, the representatives must belong to a prescribed
closed set $S\subset\R^D$. Pandey and Roychowdhury determine optimal
codebooks and errors for uniform distributions with line and circle
constraints~\cite{pandey2023}, and for the Cantor distribution under several
constraints~\cite{pandey2024constrained}. These explicit configurations show
how the constraint affects both the limiting error and the rate at which
it is approached. When the source is separated from $S$, that limit is
positive, so the asymptotic problem can be formulated as estimating the
difference between the quantization error and its limit. For an upper
estimate, choose finitely many centers that approximate the nearest points
on $S$. The triangle inequality bounds the distance from a source point
to these centers by its distance to $S$ plus the approximation error of
its nearest point. For a lower estimate, one instead controls how much
the source-to-center distance increases as the center moves away from
a nearest point. Integrating this increase over a suitable part of the
source gives a lower bound for the excess.

Our first main result is an upper comparison based on nearest-point projection.
A Borel nearest-point selector transfers ordinary quantization estimates for
the projected source measure to upper bounds for the constrained excess.
This comparison holds for arbitrary nonempty closed constraints, without
requiring unique nearest points, continuity of the selector, or local
concentration estimates.  At infinite order, we obtain
an upper bound by covering the full nearest-point image of the source support.

For lower bounds, we allow the power-growth coefficient to vary with the
source point and incorporate it into a weighted projected measure.
An $A_p$ condition on its density gives a comparison with ordinary
quantization errors of a reference probability measure, without imposing
small-ball regularity on that measure.  Under uniform growth
and an upper ball-mass estimate, the power-law lower bound also follows
from the general distortion theorem of Riegler, Koliander and
B\"olcskei~\cite{riegler2023lossy}; see Example~\ref{ex:growth}.
The weighted formulation applies to convex and uniformly prox-regular
constraints and allows growth weights with no positive uniform lower bound.
For the smaller concentric disk and the lifted Cantor distribution,
we obtain $\Delta_{n,2}=e_{n,2}^2-L_2^2\asymp n^{-2/d}$, where $d$
is the Hausdorff dimension of the projected support.

\subsection{Notation and main results}\label{section:main_result}
Throughout the article, $P$ is a Borel probability with nonempty compact
support $K\subset\R^D$, and $S\subset\R^D$ is nonempty and closed, except
in the finite-moment existence statement. Write $\rho(x,y)=|x-y|$ and
$\rho(x,A)=\inf_{a\in A}|x-a|$ for a nonempty set $A$.
For $n\ge1$, put
\[
\Acal_n(S)=\{\alpha\subset S:1\le|\alpha|\le n\}.
\]
An \emph{admissible codebook} for $e_{n,r}(P;S)$ means a nonempty finite
set $\alpha\subset S$ with $|\alpha|\le n$, that is,
$\alpha\in\Acal_n(S)$. For a different constraint $A$, admissibility
means membership in $\Acal_n(A)$.
For $0<r<\infty$ define
\begin{equation}\label{eq:errors}
F_r(\alpha)=\left(\int_K\rho(x,\alpha)^r\,dP(x)\right)^{1/r},
\qquad e_{n,r}(P;S)=\inf_{\alpha\in\Acal_n(S)}F_r(\alpha).
\end{equation}
For $r=\infty$, use $F_\infty(\alpha)=\max_{x\in K}\rho(x,\alpha)$.
The collection of minimizers is denoted by $C_{n,r}(P;S)$.
When the source and constraint are fixed, we write $e_{n,r}$.

The metric projection and distance baseline are
\begin{equation}\label{eqn:projection}
\begin{gathered}
b(x)=\rho(x,S),\qquad
\pi_S(x)=\{a\in S:|x-a|=b(x)\},\qquad
E=\pi_S(K)=\bigcup_{x\in K}\pi_S(x),\\
L_r=\left(\int_K b(x)^r\,dP(x)\right)^{1/r},\qquad B=\max_K b.
\end{gathered}
\end{equation}
The set $E$ is compact, and there is a Borel selector $T:K\to E$ with
$T(x)\in\pi_S(x)$; both assertions are proved in
Section~\ref{section:preliminary_topology}. Fix one such selector and put
\[
\mu=T_*P,\qquad E_\mu=\supp{\mu}.
\]
We call $\mu$ the \emph{projected measure}. For ordinary quantization write
\[
q_{n,r}(\mu;A)=\inf_{\alpha\in\Acal_n(A)}
 \left(\int\rho(y,\alpha)^r\,d\mu(y)\right)^{1/r},
\qquad q_{n,r}(\mu)=q_{n,r}(\mu;\R^D).
\]
For a nonempty compact set $A$, write
\[
c_n(A;A)=\inf_{\alpha\in\Acal_n(A)}\sup_{y\in A}\rho(y,\alpha),
\qquad c_n(A)=\inf_{\alpha\in\Acal_n(\R^D)}\sup_{y\in A}\rho(y,\alpha).
\]
Thus $c_n(A)=e_{n,\infty}(A)$ in the usual covering-radius notation.

Lemma~\ref{prop:baseline} proves $e_{n,r}\downarrow L_r$ and
$e_{n,\infty}\downarrow B$. For $0<r<\infty$, define
\begin{equation}\label{eqn:defn_generalized_e_n_r}
 \widetilde e_{n,r}=e_{n,r}-L_r,\qquad
 \Delta_{n,r}=e_{n,r}^r-L_r^r,\qquad
 \widehat e_{n,r}=\Delta_{n,r}^{1/r}.
\end{equation}
At infinite order we use $\widetilde e_{n,\infty}=e_{n,\infty}-B$.
For a positive excess tending to zero, define
\begin{align}
\ldimqhat{r}{P;S}
 &=\liminf_{n\to\infty}\frac{\log n}{-\log\widehat e_{n,r}},&
\udimqhat{r}{P;S}
 &=\limsup_{n\to\infty}\frac{\log n}{-\log\widehat e_{n,r}},
 \label{eqn:udimhat}\\
\ldimqtilde{r}{P;S}
 &=\liminf_{n\to\infty}\frac{\log n}{-\log\widetilde e_{n,r}},&
\udimqtilde{r}{P;S}
 &=\limsup_{n\to\infty}\frac{\log n}{-\log\widetilde e_{n,r}}.
 \label{eqn:udim_ldim_tilde}
\end{align}
If $e_{n_0,r}=L_r$ for some integer $n_0\ge1$, then
$\Acal_{n_0}(S)\subset\Acal_n(S)$ for every $n\ge n_0$, and hence
\[
L_r\le e_{n,r}\le e_{n_0,r}=L_r.
\]
Thus $\widetilde e_{n,r}=\Delta_{n,r}=\widehat e_{n,r}=0$ for every
$n\ge n_0$. In this case, set
\[
 \ldimqtilde r{P;S}=\udimqtilde r{P;S}
 =\ldimqhat r{P;S}=\udimqhat r{P;S}=0.
\]
If $e_{n_0,\infty}=B$, then $e_{n,\infty}=B$ for $n\ge n_0$;
set $\ldimqtilde\infty{P;S}=\udimqtilde\infty{P;S}=0$. For ordinary errors $q_{n,r}(\mu)$, we write
$\Dlo_r(\mu)$ and $\Dup_r(\mu)$ for the lower and upper dimensions
defined by the same logarithmic limits and the same zero convention. Hausdorff and box dimensions, including
the full-mass dimensions of a measure, are specified in
Section~\ref{app_supp_notation_definitions}.

\begin{theorem}[Projection upper estimates]\label{thm:upper_asymptotics}
Let $T$ be any Borel nearest-point selector and $\mu=T_*P$.
Set $M_0=1+B+\diam E$. For every $n\ge1$ and $1\le r<\infty$,
\begin{align}
\widetilde e_{n,r}&\le q_{n,r}(\mu;E_\mu)\le2q_{n,r}(\mu),
 \label{eq:additive-upper}\\
\Delta_{n,r}&\le rM_0^{r-1}q_{n,1}(\mu;E_\mu)
 \le2rM_0^{r-1}q_{n,1}(\mu).\label{eq:power-upper}
\end{align}
For $0<r<1$,
\begin{equation}\label{eq:subunit-upper}
\Delta_{n,r}\le q_{n,r}(\mu;E_\mu)^r\le2^rq_{n,r}(\mu)^r.
\end{equation}
At infinite order,
\begin{equation}\label{eq:infty-upper}
0\le\widetilde e_{n,\infty}\le c_n(E;E)\le2c_n(E).
\end{equation}
In particular, for $1\le r<\infty$,
\begin{align}
\ldimqtilde{r}{P;S}&\le\Blo^*\mu\le\Blo E,&
\udimqtilde{r}{P;S}&\le\Bup^*\mu\le\Bup E,
\label{eq:corrected-tilde-dim}\\
\frac1r\ldimqhat{r}{P;S}&\le\Blo^*\mu\le\Blo E,&
\frac1r\udimqhat{r}{P;S}&\le\Bup^*\mu\le\Bup E.
\label{eq:corrected-hat-dim}
\end{align}
All constants in the error estimates are independent of $n$.
\end{theorem}

For the lower estimate, fix $R\ge1$ with $K\subset\overline B(0,R)$ and
$a_0\in S$, and put
\begin{equation}\label{eq:common-compact}
C=S\cap\overline B(0,2R+|a_0|+1),\qquad
H=\max\{|x-a|:x\in K,\ a\in C\}.
\end{equation}
Replacing a center outside $C$ by $a_0$ decreases its distance from every
point of $K$. Thus restricting all codebooks to $C$ leaves $e_{n,r}$
unchanged; the details are given in Lemma~\ref{prop:existence}.
For the lower estimates, the weighted growth condition is
\begin{equation}\label{eq:weighted-main-growth}\tag{L1}
\begin{gathered}
 1\le r<\infty,\quad\beta>0,\quad F\subset K\text{ is Borel},\\
 w:F\to[0,\infty)\text{ is Borel},\qquad \int_Fw\,dP<\infty,\\
 |x-a|^r-b(x)^r\ge w(x)|a-T(x)|^\beta
 \qquad(x\in F,\ a\in C).
\end{gathered}
\end{equation}
The set $F$ and the weight $w$ are fixed independently of $n$ and
of the codebook $\alpha\in\Acal_n(C)$.

The finite-measure cost $\Vcal_{n,\beta}$ and the weight constant
$[\wh]_{A_p(\eta)}$ are defined in \eqref{eq:V} and \eqref{eq:Ap}.

\begin{theorem}[Weighted projection lower estimates]\label{thm:weighted_lower_main}
Suppose \eqref{eq:weighted-main-growth} holds for $r,\beta,F,w$, with $C$ as in
\eqref{eq:common-compact}.
Set $\nu_F=T_*(wP|_F)$. Suppose
$d\nu_F=\wh\,d\eta$ for a compactly supported Borel probability
$\eta$, where $\wh\in A_p(\eta)$ and $1<p<\infty$. Put
\[
 s=\beta/p,\qquad W=\nu_F(\R^D),\qquad
 \kappa=\frac{W}{[\wh]_{A_p(\eta)}},\qquad
 M_0=1+B+\diam E.
\]
Then $0<\kappa<\infty$ and, for every $n\ge1$,
\begin{align}
 \Delta_{n,r}
 &\ge\Vcal_{n,\beta}(\nu_F;C)
 \ge\Vcal_{n,\beta}(\nu_F)
 \ge\kappa q_{n,s}(\eta)^\beta,
 \label{eq:weighted-main-power}\\
 \widetilde e_{n,r}
 &\ge\frac{\kappa}{rM_0^{r-1}}q_{n,s}(\eta)^\beta,
 &\widehat e_{n,r}
 &\ge\kappa^{1/r}q_{n,s}(\eta)^{\beta/r}.
 \label{eq:weighted-main-errors}
\end{align}
Consequently,
\begin{align}
 \ldimqtilde{r}{P;S}&\ge\frac1\beta\Dlo_s(\eta),&
 \udimqtilde{r}{P;S}&\ge\frac1\beta\Dup_s(\eta),
 \label{eq:weighted-main-tilde-dim}\\
 \ldimqhat{r}{P;S}&\ge\frac r\beta\Dlo_s(\eta),&
 \udimqhat{r}{P;S}&\ge\frac r\beta\Dup_s(\eta).
 \label{eq:weighted-main-hat-dim}
\end{align}
If, in addition, $d,M>0$ and
$\eta(B(y,t))\le Mt^d$ for every $y\in\R^D,t>0$, then
\begin{equation}\label{eq:weighted-main-rate}
 \Delta_{n,r}\ge\kappa\left(\frac{d}{d+\beta/p}\right)^p
                  M^{-\beta/d}n^{-\beta/d},\qquad
 \ldimqtilde{r}{P;S}\ge\frac d\beta,\qquad
 \ldimqhat{r}{P;S}\ge\frac{rd}\beta.
\end{equation}
All constants are independent of $n$.
\end{theorem}

The upper ball condition used in Corollary~\ref{thm:lower_asymptotics} is that there are $d,M>0$ such that
\begin{equation}\label{eq:upper-mass}
\mu(B(y,t))\le Mt^d\qquad(y\in\R^D,\ t>0).
\end{equation}
The usual upper Ahlfors estimate, stated for centers in $E_\mu$ and small
radii, implies \eqref{eq:upper-mass} after a change of constant
(Example~\ref{ex:growth}). The distance-growth condition is
\begin{equation}\label{cond:Gbeta}\tag{L2}
\begin{gathered}
\text{There are }\beta,b_0,c>0\text{ and a Borel set }F\subset K,
\quad m=P(F)>0,\\
b(x)\ge b_0,\qquad
|x-a|-b(x)\ge c|a-T(x)|^\beta
\quad(x\in F,\ a\in C).
\end{gathered}
\end{equation}
The constants $b_0,c,\beta$ and the set $F$ are independent of $n$
and of the codebook $\alpha\in\Acal_n(C)$.
For a closed convex constraint separated from $K$, the nearest-point
inequality gives
\[
|x-a|^2-b(x)^2\ge|a-T(x)|^2,
\qquad |x-a|-b(x)\ge\frac{|a-T(x)|^2}{2H}
\quad(a\in C).
\]
Consequently \eqref{cond:Gbeta} holds with $F=K$ and $\beta=2$ in this
geometric class. Section~\ref{sec:convex-growth} verifies these inequalities.

If $L_r>0$, Lemma~\ref{lem:conversion} gives
$\widehat e_{n,r}\asymp\widetilde e_{n,r}^{1/r}$.
Lemma~\ref{lem:log-transfer}(ii) therefore yields
\begin{equation}\label{eq:hat-tilde-dimension}
 \ldimqhat r{P;S}=r\ldimqtilde r{P;S},\qquad
 \udimqhat r{P;S}=r\udimqtilde r{P;S}.
\end{equation}

If $E$ is Ahlfors regular of dimension $d$, its covering radii are at most
$C_E n^{-1/d}$. Together with Theorem~\ref{thm:upper_asymptotics}
and Corollary~\ref{thm:lower_asymptotics}, this gives, under
\eqref{eq:upper-mass} and \eqref{cond:Gbeta} with the same $d$,
\begin{align*}
\frac d\beta\le\ldimqtilde{r}{P;S}
 &\le\udimqtilde{r}{P;S}\le d,\\
\frac{rd}\beta\le\ldimqhat{r}{P;S}
 &\le\udimqhat{r}{P;S}\le rd\qquad(1\le r<\infty).
\end{align*}
The geometric examples below include matching rates with $\beta=2$.
They exhibit additive dimensions $d/2$ and hatted dimensions $rd/2$.

Section~\ref{section:preliminary} establishes the projection, existence,
and excess-comparison results. Section~\ref{section:upper_quantization_dimension_estimate}
proves the projection upper estimates and the corresponding dimension
bounds. Section~\ref{section:lower_quantization_dimension_estimate}
develops the weighted projection lower comparison, derives the
positive-mass growth corollary, and proves
Theorem~\ref{thm:weighted_lower_main}. Its examples treat prox-regular
constraints, uniform growth and regular dimensions, weighted Zador
coefficients, and growth on a restricted set of centers.
Section~\ref{section:examples} applies the estimates to convex constraints,
circles, and a lifted Cantor distribution. The appendix supplies the
dimension, covering, and uniform-growth estimates used in the proofs.

\section{Preliminary Results}\label{section:preliminary}
\subsection{The projection and a Borel selector}\label{section:preliminary_topology}
Two inverse-image operations for the set-valued projection will be useful.
For $A\subset E$, put
\begin{equation}\label{eqn:reversal}
H(A)=\{x\in K:\pi_S(x)\cap A\ne\varnothing\},\qquad
L(A)=\{x\in K:\pi_S(x)\subset A\}.
\end{equation}
The first operation records whether a nearest-point set meets $A$; the
second records whether all nearest points lie in $A$.

\begin{lemma}\label{lem:projection-topology}\label{compactness_pi_s_k}\label{lemma:closed_graph_property}\label{closedness_pi_inv_s_k}\label{lemma:approximation_pi_s_x_and_pi_s_x_n}
The sets $\pi_S(x)$ are nonempty and compact. The graph
\[
 G=\{(x,a)\in K\times S:|x-a|=b(x)\}
\]
and its second-coordinate projection $E$ are compact. If $A\subset E$ is
closed, then $H(A)$ is closed in $K$. If $U\subset E$ is relatively open,
then $L(U)$ is relatively open in $K$.
Moreover, if $x_j\to x$ in $K$, then
\begin{equation}\label{eq:upper-stability}
 \sup_{a\in\pi_S(x_j)}\rho(a,\pi_S(x))\longrightarrow0.
\end{equation}

\end{lemma}
\begin{proof}
First, the triangle inequality gives
\[
 b(x)\le |x-y|+b(y),\qquad b(y)\le |x-y|+b(x),
\]
so $|b(x)-b(y)|\le |x-y|$. Fix $a_0\in S$. A sequence in $S$ whose distances
from a fixed $x$ tend to $b(x)$ is bounded. A convergent subsequence has its
limit in $S$, and continuity of distance shows that the limit belongs to
$\pi_S(x)$.

Choose $R\ge1$ with $K\subset\overline B(0,R)$. For every $(x,a)\in G$,
\[
 |a|\le |x|+|x-a|=|x|+b(x)
 \le |x|+|x-a_0|\le 2R+|a_0|.
\]
The graph is closed because $S$ is closed and $(x,a)\mapsto |x-a|-b(x)$ is
continuous. This uniform bound makes the graph compact. Its projection $E$
is consequently compact. For closed $A\subset E$, the set
$G\cap(K\times A)$ is compact, and its first-coordinate projection is $H(A)$.
Thus $H(A)$ is closed. The identity
\[
 K\setminus L(U)=H(E\setminus U)
\]
proves the assertion for $L(U)$.

To prove \eqref{eq:upper-stability}, suppose that it fails. There would then
be $\varepsilon>0$, a subsequence still indexed by $j$, and
$a_j\in\pi_S(x_j)$ with $\rho(a_j,\pi_S(x))\ge\varepsilon$. Compactness of
$E$ gives a further subsequence with $a_j\to a$. The closed graph property
gives $a\in\pi_S(x)$. But
\[
 \varepsilon\le\rho(a_j,\pi_S(x))\le |a_j-a|\longrightarrow0,
\]
which is a contradiction.
\end{proof}

\begin{lemma}\label{lem:selector}\label{lemma:pi_s_k_single_valued}
There exists a Borel map $T:K\to E$ satisfying $T(x)\in\pi_S(x)$ for every $x\in K$.
\end{lemma}
\begin{proof}
Lemma~\ref{lem:projection-topology} gives nonempty compact values $\pi_S(x)\subset E$, and shows that
$H(Q)=\{x\in K:\pi_S(x)\cap Q\ne\varnothing\}$ is closed for closed $Q\subset E$.
For a relatively open set $U\subset E$ with $E\setminus U\ne\varnothing$, set
\[
 Q_j=\{y\in E:\dist(y,E\setminus U)\ge1/j\}.
\]
These sets are compact and $U=\bigcup_jQ_j$. Consequently,
\[
 \{x\in K:\pi_S(x)\cap U\ne\varnothing\}
 =\bigcup_{j\ge1}H(Q_j)\in\mathcal B(K).
\]
For $U=E$ the inverse image is $K$. Thus $x\mapsto\pi_S(x)$ is a measurable closed-valued mapping from $(K,\mathcal B(K))$ into $\R^D$. Rockafellar--Wets \cite[Corollary~14.6]{rockafellar2009variational} supplies the required Borel selection.
\end{proof}

For any finite Borel measure $\lambda$ on $K$ and any nonnegative Borel function $g$,
\begin{equation}\label{eq:pushforward-finite}
 \int_{\R^D}g(y)\,d(T_*\lambda)(y)=\int_Kg(T(x))\,d\lambda(x).
\end{equation}
In particular, for $\lambda=P$ and $\mu=T_*P$,
\begin{equation}\label{eq:pushforward-integral}
 \int_{\R^D}g(y)\,d\mu(y)=\int_Kg(T(x))\,dP(x).
\end{equation}

\subsection{Support and local mass of the projected measure}
For identifying the support with the full projection set, we use the
following condition:
\begin{equation}\label{cond_A}\tag{U1}
P\{x\in K:\pi_S(x)\subset B(y,t)\}>0
\qquad(y\in E,\ t>0).
\end{equation}

\begin{proposition}\label{prop:support}\label{lemma:support_pushforward_measure}
For any Borel nearest-point selector,
\begin{equation}\label{eq:support-inclusions}
 E_\mu\subset\overline{T(K)}\subset E.
\end{equation}
Under U1, both inclusions are equalities. If $T$ is continuous, then $E_\mu=T(K)$ and this image is compact.
\end{proposition}
\begin{proof}
If $y\notin\overline{T(K)}$, there is a ball $B(y,t)$ disjoint from $T(K)$.
Its inverse image is empty, so $\mu(B(y,t))=0$ and $y\notin E_\mu$. This proves
the first inclusion. The second uses $T(K)\subset E$ and the closedness of $E$.

Under U1, for every $y\in E$ and $t>0$,
\begin{align*}
 \mu(B(y,t))
 &=P(\{x:T(x)\in B(y,t)\})\\
 &\ge P(\{x:\pi_S(x)\subset B(y,t)\cap E\})>0.
\end{align*}
Thus every $y\in E$ belongs to $E_\mu$, which completes both equalities.

If $T$ is continuous, fix $y=T(x)\in T(K)$ and $t>0$. The inverse image
$T^{-1}(B(y,t))$ is a relatively open neighborhood of $x$ in $K$.
Since $K=\supp P$, that neighborhood has positive $P$-mass. Hence $y\in E_\mu$,
and the reverse inclusion has already been proved.
\end{proof}

A quantitative version of \eqref{cond_A} is the local condition
\begin{equation}\label{cond_D}\tag{U2}
P\{x\in K:\pi_S(x)\subset B(y,t)\}\ge ct^d
\qquad(y\in E,\ 0<t<t_0),
\end{equation}
where $c,t_0>0$ and $d\ge0$ are fixed.

\begin{proposition}\label{prop:U2}\label{lemma:lower_ahlfors_reg_T_*P}
Condition \eqref{cond_D} implies U1 and, for every Borel selector,
\begin{equation}\label{eq:U2-transfer}
 \mu(B(y,t))\ge ct^d\qquad(y\in E,\ 0<t<t_0).
\end{equation}
In particular, $E_\mu=E$.
\end{proposition}
\begin{proof}
If $0<t<t_0$, the containment event in \eqref{cond_D} is included in
$\{x:T(x)\in B(y,t)\}$. Taking probabilities gives
\eqref{eq:U2-transfer}. For any larger ball choose
$u\in(0,\min\{t,t_0\})$; its containment event contains the event for
$B(y,u)$, which has positive mass. This proves U1 at every radius, and
Proposition~\ref{prop:support} gives $E_\mu=E$.
\end{proof}

\subsection{Voronoi partitions, attainment, and limiting errors}\label{section:basic_properties_quantization}
\label{section:preliminary_constrained_quantization}
For measurable functions $d_1,\ldots,d_m:X\to[0,\infty)$ on a measurable
space $X$, ties are assigned to the smallest minimizing index. Set
\[
 V_i=\bigcap_{j<i}\{d_i<d_j\}\cap\bigcap_{j>i}\{d_i\le d_j\}.
\]
The comparisons are measurable, and each $x$ has a unique smallest
minimizing index. Thus $(V_i)$ is a measurable partition and
\[
 \min_i d_i(x)=\sum_i d_i(x)\one_{V_i}(x).
\]
For an ordered codebook $\alpha=\{a_1,\ldots,a_m\}$, take $X=K$ and
$d_i(x)=|x-a_i|$. Then
\begin{equation}\label{eq:voronoi}
 V_i=\{x\in K:|x-a_i|<|x-a_j|\ (j<i),\quad
                       |x-a_i|\le|x-a_j|\ (j>i)\},
\end{equation}
and
\begin{equation}\label{eq:partition-cost}
 \int_K\rho(x,\alpha)^r\,dP(x)
 =\sum_{i=1}^m\int_{V_i}|x-a_i|^r\,dP(x)\qquad(r>0).
\end{equation}
More generally, for a finite nonnegative measurable distortion $\varrho$
on $X\times A$, put $d_i(x)=\varrho(x,a_i)$ and
$q_\alpha(x)=a_i$ on $V_i$. Then $q_\alpha$ is measurable and
$\varrho(x,q_\alpha(x))=\min_{a\in\alpha}\varrho(x,a)$.
Finite-moment existence is part of the classical quantization theory
\cite{graf}. The following argument treats a closed constraint directly and
also gives a compact set containing a minimizing codebook at every order
when the source is compact.

\begin{lemma}\label{prop:existence}\label{lemma:existence}
Let $S\subset\R^D$ be nonempty and closed. If $P$ is a Borel probability with
$\int |x|^r\,dP(x)<\infty$ for some $0<r<\infty$, then
$C_{n,r}(P;S)\ne\varnothing$ for every $n\ge1$.
If $K=\supp P$ is compact, then $C_{n,\infty}(P;S)\ne\varnothing$.
\end{lemma}
\begin{proof}
\textit{Finite order.} Fix $a_0\in S$. The inequality
\[
 |x-a_0|^r\le 2^{\max\{r-1,0\}}(|x|^r+|a_0|^r)
\]
shows that $e_{n,r}^{r}$ is finite. Choose $\alpha_j\in\Acal_n(S)$ with
\[
 e_{n,r}^{r}\le\int\rho(x,\alpha_j)^r\,dP(x)
 \le e_{n,r}^{r}+j^{-1}.
\]
Represent $\alpha_j$ by an $n$-tuple $(a_{j,1},\ldots,a_{j,n})$, repeating
points when necessary. We construct nested infinite index sets
$I_0\supset I_1\supset\cdots\supset I_n$, beginning with $I_0=\N$.
Suppose that $I_{i-1}$ has been chosen. If
$(a_{j,i})_{j\in I_{i-1}}$ has a bounded subsequence, the
Bolzano--Weierstrass theorem gives an infinite subset $I_i\subset I_{i-1}$
and $a_i\in\R^D$ such that $a_{j,i}\to a_i$ as $j\to\infty$
through $I_i$. If there is no bounded subsequence, then for every
$R'>0$ the set
\[
\{j\in I_{i-1}:|a_{j,i}|\le R'\}
\]
is finite; otherwise it would supply a bounded subsequence. In this
case set $I_i=I_{i-1}$, so $|a_{j,i}|\to\infty$ along $I_i$.
Restriction to $I_i$ preserves every convergence or divergence already
obtained for coordinates $1,\ldots,i-1$.

After the $n$th step, enumerate $I_n$ increasingly as
$j_1<j_2<\cdots$. For each $1\le i\le n$, either
$a_{j_k,i}\to a_i\in\R^D$ or $|a_{j_k,i}|\to\infty$.
Also $j_k^{-1}\to0$, so the integrals of the corresponding codebooks
still tend to $e_{n,r}^r$. In the rest of the proof we denote this
subsequence again by $\alpha_j$ and its coordinates by $a_{j,i}$.

At least one coordinate has a finite limit. Otherwise
$\rho(x,\alpha_j)\to\infty$ for every fixed $x$, and Fatou's lemma would give
\[
 \infty=\int\liminf_j\rho(x,\alpha_j)^r\,dP(x)
 \le\liminf_j\int\rho(x,\alpha_j)^r\,dP(x)=e_{n,r}^{r}<\infty.
\]
Let $I$ be the nonempty set of coordinates with finite limits, denote those
limits by $a_i$, and put $\alpha=\{a_i:i\in I\}$. Closedness of $S$ gives
$\alpha\in\Acal_n(S)$. For each fixed $x$, the distances to the coordinates
outside $I$ tend to infinity, while the remaining distances converge.
Taking a minimum over finitely many coordinates consequently gives
$\rho(x,\alpha_j)\to\rho(x,\alpha)$. A second application of Fatou yields
\[
 e_{n,r}^{r}\le\int\rho(x,\alpha)^r\,dP(x)
 \le\liminf_j\int\rho(x,\alpha_j)^r\,dP(x)=e_{n,r}^{r}.
\]
Thus $\alpha$ is a minimizer.

\textit{Compact source and infinite order.} Choose $R\ge1$ with
$K\subset\overline B(0,R)$, and put
\begin{equation}\label{eq:localization}
 C=S\cap\overline B(0,2R+|a_0|+1).
\end{equation}
If $a\in S\setminus C$ and $x\in K$, then
\[
 |x-a|\ge |a|-R>R+|a_0|\ge |x-a_0|.
\]
Replace every center outside $C$ by $a_0$. This does not increase the
pointwise distance or the number of centers. Hence the infimum defining
$e_{n,\infty}$ is the same when all centers are required to lie in $C$.
The set $C^n$ is compact. If $\alpha=\{a_1,\ldots,a_n\}$ and
$\beta=\{b_1,\ldots,b_n\}$, then
\begin{equation}\label{eq:distance-stability}
 |\rho(x,\alpha)-\rho(x,\beta)|
 \le \max_{1\le i\le n}|a_i-b_i|\qquad(x\in\R^D).
\end{equation}
Define
\[
\Phi(a_1,\ldots,a_n)=\max_{x\in K}\min_{1\le i\le n}|x-a_i|
\qquad((a_1,\ldots,a_n)\in C^n).
\]
By \eqref{eq:distance-stability},
\[
|\Phi(a_1,\ldots,a_n)-\Phi(b_1,\ldots,b_n)|
\le\max_{1\le i\le n}|a_i-b_i|.
\]
Thus $\Phi$ is continuous on the nonempty compact set $C^n$.
Choose tuples $\mathbf a_j\in C^n$ with
$\Phi(\mathbf a_j)\to\inf_{C^n}\Phi$. Compactness gives a
subsequence converging to $\mathbf a\in C^n$, and continuity yields
\[
\Phi(\mathbf a)=\lim_{k\to\infty}\Phi(\mathbf a_{j_k})
=\inf_{C^n}\Phi=e_{n,\infty}.
\]
The last equality follows because every codebook in $\Acal_n(C)$ can
be represented by an $n$-tuple, repeating a point when necessary,
and every $n$-tuple defines a member of $\Acal_n(C)$.
The distinct coordinates of $\mathbf a$ therefore form a minimizing
codebook. The same replacement is available for every
finite order when the source is compact.
\end{proof}

\begin{lemma}\label{prop:baseline}\label{lemma:e_infty_r_by_projection_map}
For compact $K$ and closed nonempty $S$,
\begin{equation}\label{eq:baseline-limit}
 e_{n,r}\downarrow L_r\quad(0<r<\infty),\qquad
 e_{n,\infty}\downarrow B.
\end{equation}
\end{lemma}
\begin{proof}
The admissible classes increase with $n$, so each error sequence is
nonincreasing. Every admissible codebook satisfies
$\rho(x,\alpha)\ge b(x)$, giving the stated lower bounds.

Fix $\varepsilon\in(0,1)$. Compactness of $E$ gives a finite
$\varepsilon$-net $\beta_\varepsilon\subset E$, i.e., a finite set with
\[
E\subset\bigcup_{a\in\beta_\varepsilon}\overline B(a,\varepsilon).
\]
For every $x\in K$, the set $\pi_S(x)$ is nonempty by
Lemma~\ref{lem:projection-topology}, so choose $p\in\pi_S(x)$.
Since $p\in E$, the displayed cover gives $a\in\beta_\varepsilon$
such that $|p-a|\le\varepsilon$. Then
\begin{equation}\label{eq:net-baseline}
 b(x)\le\rho(x,\beta_\varepsilon)
 \le |x-p|+|p-a|\le b(x)+\varepsilon.
\end{equation}
For $n\ge|\beta_\varepsilon|$ this gives
\[
 L_r^r\le e_{n,r}^{r}\le\int(b+\varepsilon)^r\,dP,
 \qquad B\le e_{n,\infty}\le B+\varepsilon.
\]
The functions $(b+\varepsilon)^r$ converge to $b^r$ and are bounded by
$(B+1)^r$. Dominated convergence, followed by the continuity of the
$r$th root, proves the finite-order limit. Letting $\varepsilon\downarrow0$
in the supremum estimate proves the infinite-order limit.
\end{proof}

\subsection{Optimal codebooks and dependence on the order}

For nonempty compact sets $A$ and $A'$, we use
\[
d_H(A,A')=\max\{\sup_{a\in A}\rho(a,A'),\ \sup_{a'\in A'}\rho(a',A)\}.
\]
For each fixed integer $n\ge1$, this metric gives the following infimum formula.

\begin{proposition}\label{prop:hausdorff-inf}\label{lemma:e_n_infty_and_hausdorff_10.4}
For nonempty compact $K$ and nonempty closed $S$,
\begin{equation}\label{eq:hausdorff-inf}
 e_{n,\infty}(K;S)=\inf_{\alpha\in\Acal_n(S)}d_H(\alpha,K).
\end{equation}
\end{proposition}
\begin{proof}
For every $\alpha$, $\sup_{x\in K}\rho(x,\alpha)\le d_H(\alpha,K)$, so the
left side is no greater than the right side. For the other direction, take
any $\alpha\in\Acal_n(S)$ and retain exactly those centers that are nearest
to at least one point of $K$. Call the retained set $\beta$. It is nonempty,
and $\rho(x,\beta)=\rho(x,\alpha)$ for every $x\in K$. If $a\in\beta$,
choose $x_a\in K$ with $|x_a-a|=\rho(x_a,\alpha)$. Then
\[
 \rho(a,K)\le |a-x_a|\le\sup_{x\in K}\rho(x,\alpha).
\]
Both directed terms in $d_H(\beta,K)$ are bounded by this last supremum.
Therefore
\[
 \inf_{\gamma\in\Acal_n(S)}d_H(\gamma,K)
 \le d_H(\beta,K)\le\sup_{x\in K}\rho(x,\alpha).
\]
Taking the infimum over the original $\alpha$ gives the remaining inequality.
\end{proof}

\begin{proposition}\label{prop:cardinality}\label{lemma:existence_previously}
In any of the existence settings of Lemma~\ref{prop:existence}, and
for $n\ge2$,
\begin{equation}\label{eq:cardinality}
 e_{n,r}<e_{n-1,r}
 \quad\Longleftrightarrow\quad
 |\alpha|=n\ \hbox{for every }\alpha\in C_{n,r}(P;S).
\end{equation}
This includes $r=\infty$ for compact source. At finite order, strict decrease
also implies that every cell of an optimal codebook has positive mass.
\end{proposition}
\begin{proof}
If strict decrease holds and an optimal $\alpha$ has at most $n-1$ points,
then $e_{n-1,r}\le F_r(\alpha)=e_{n,r}$, a contradiction. Conversely, if
$e_{n,r}=e_{n-1,r}$, an $(n-1)$-optimal codebook exists and is also
$n$-optimal. It has at most $n-1$ points.

At finite order, suppose an optimal codebook has a zero-mass cell $V_i$.
Removing $a_i$ leaves the distance unchanged outside $V_i$, because every
point outside that cell has another nearest center. Hence removing $a_i$
leaves the integral unchanged. This would imply
$e_{n-1,r}=e_{n,r}$, contradicting strict decrease.
\end{proof}

\begin{proposition}\label{prop:order}\label{lemma:monotone_e_n_infty_e_n_r_constrained}\label{lemma:monotone_e_n_infty_e_n_r_constrained_10.1}
For fixed $n$ and $0<r\le s<\infty$,
\begin{equation}\label{eq:order-comparisons}
 e_{n,r}\le e_{n,s},\qquad L_r\le L_s,
 \qquad \widehat e_{n,r}\le\widehat e_{n,s}.
\end{equation}
The inequality $e_{n,r}\le e_{n,s}$ also holds with $s=\infty$.
\end{proposition}
\begin{proof}
First, put $p=r/s\in(0,1]$. For a nonnegative random variable $Z$ with a
finite $s$th moment, Jensen's inequality for the concave function $t^p$ gives
\[
 \int Z^r\,dP=\int(Z^s)^p\,dP
 \le\left(\int Z^s\,dP\right)^p.
\]
Taking the $r$th root yields $\|Z\|_r\le\|Z\|_s$.
Apply this to $Z=\rho(\cdot,\alpha)$ for each fixed admissible $\alpha$.
Given $\varepsilon>0$, choose $\alpha$ with
$F_s(\alpha)\le e_{n,s}+\varepsilon$. Then
$e_{n,r}\le F_r(\alpha)\le F_s(\alpha)\le e_{n,s}+\varepsilon$.
Letting $\varepsilon\downarrow0$ proves the first comparison. Applying the
same norm inequality to $Z=b$ proves the second. For $s=\infty$, use
$\|Z\|_r\le\sup_K Z$ and the same admissible-codebook argument.

For the comparison of $\widehat e_{n,r}$ and $\widehat e_{n,s}$,
the definitions give
\begin{equation}\label{eq:inf-subtraction}
 \Delta_{n,r}
 =\inf_{\alpha\in\Acal_n(S)}
       \int_K\bigl(\rho(x,\alpha)^r-b(x)^r\bigr)\,dP(x).
\end{equation}
Indeed, put $I_r$ equal to the infimum on the right. For every
$\alpha\in\Acal_n(S)$,
\[
\int_K(\rho(x,\alpha)^r-b(x)^r)\,dP(x)
=F_r(\alpha)^r-L_r^r\ge e_{n,r}^r-L_r^r=\Delta_{n,r},
\qquad I_r\ge\Delta_{n,r}.
\]
For $\varepsilon>0$, choose $\alpha_\varepsilon\in\Acal_n(S)$
with $F_r(\alpha_\varepsilon)<e_{n,r}+\varepsilon$. Then
\[
I_r\le F_r(\alpha_\varepsilon)^r-L_r^r
\le(e_{n,r}+\varepsilon)^r-L_r^r
\longrightarrow\Delta_{n,r}\qquad(\varepsilon\downarrow0).
\]
Consequently $\Delta_{n,r}\le I_r\le\Delta_{n,r}$,
which proves \eqref{eq:inf-subtraction}.

For $u\ge v\ge0$, subadditivity of $t^p$ gives
\[
 u^r=((u^s-v^s)+v^s)^p
 \le (u^s-v^s)^p+v^r.
\]
Apply this with $u=\rho(x,\alpha)$ and $v=b(x)$. If
$J_t(\alpha)=\int(\rho(x,\alpha)^t-b(x)^t)\,dP(x)$, then Jensen's
inequality gives
\[
 J_r(\alpha)
 \le\int(\rho(x,\alpha)^s-b(x)^s)^p\,dP(x)
 \le J_s(\alpha)^p.
\]
Choose $\alpha$ with $J_s(\alpha)\le\Delta_{n,s}+\varepsilon$. By
\eqref{eq:inf-subtraction},
\[
 \Delta_{n,r}\le J_r(\alpha)
 \le(\Delta_{n,s}+\varepsilon)^{r/s}.
\]
Let $\varepsilon\downarrow0$ and take the $r$th root. This proves the last
inequality in \eqref{eq:order-comparisons}, including zero excesses.
\end{proof}

\begin{proposition}\label{prop:endpoint}\label{lemma:monotone_errors_of_e_n_infty_e_n_r_constrained}
For fixed $n$ and compact source,
\begin{equation}\label{eq:endpoint-valid}
 \lim_{r\to\infty}e_{n,r}=e_{n,\infty},\qquad
 \lim_{r\to\infty}L_r=B,\qquad
 \lim_{r\to\infty}\widetilde e_{n,r}=e_{n,\infty}-B.
\end{equation}
The finite limit $\lim_{r\to\infty}\widehat e_{n,r}$ always exists and is at
most $e_{n,\infty}$. If $e_{n,\infty}>B$, then
\begin{equation}\label{eq:hatted-gap-limit}
 \lim_{r\to\infty}\widehat e_{n,r}=e_{n,\infty}.
\end{equation}
\end{proposition}
\begin{proof}
Write $a=\lim_{r\to\infty}e_{n,r}$, which exists by monotonicity and satisfies
$a\le e_{n,\infty}<\infty$. Choose $r_j\to\infty$ and minimizing codebooks
$\alpha_j$ in the common compact set $C$ of \eqref{eq:common-compact}.
Represent them by $n$-tuples and pass to a subsequence whose coordinates
converge. Their limit gives $\alpha\in\Acal_n(C)$, with
$d_H(\alpha_j,\alpha)\to0$.

Fix an integer $q\ge1$. For all sufficiently large $j$, $r_j\ge q$, and
\[
 e_{n,r_j}=F_{r_j}(\alpha_j)\ge F_q(\alpha_j).
\]
Distance functions differ pointwise by at most $d_H(\alpha_j,\alpha)$.
Minkowski's inequality, applied in both directions, therefore gives
\[
 |F_q(\alpha_j)-F_q(\alpha)|
 \le \|\rho(\cdot,\alpha_j)-\rho(\cdot,\alpha)\|_{L^q(P)}
 \le d_H(\alpha_j,\alpha)\longrightarrow0.
\]
Thus $a\ge F_q(\alpha)$ for every integer $q\ge1$.

For any nonnegative continuous $f$ on $K$, let $M=\max_K f$.
The upper bound $\|f\|_q\le M$ is immediate. If $0<t<M$, continuity and
$K=\supp P$ give $P(\{f>t\})>0$, whence
\[
 \|f\|_q\ge t\,P(\{f>t\})^{1/q}\longrightarrow t.
\]
Taking $t\uparrow M$ proves $\|f\|_q\to M$; the case $M=0$ is direct.
Apply this first to $f=\rho(\cdot,\alpha)$ to obtain
\[
 a\ge\sup_K\rho(x,\alpha)\ge e_{n,\infty}.
\]
This proves the first limit in \eqref{eq:endpoint-valid}. Applying the same
argument to $f=b$ proves $L_r\to B$. Consequently,
\[
 \lim_{r\to\infty}\widetilde e_{n,r}
 =\lim_{r\to\infty}(e_{n,r}-L_r)=e_{n,\infty}-B.
\]

By Proposition~\ref{prop:order}, the error $\widehat e_{n,r}$ is nondecreasing in $r$,
and $0\le\widehat e_{n,r}\le e_{n,r}\le e_{n,\infty}$. Its finite limit
therefore exists. Suppose now that $e_{n,\infty}>B$. Choose $\theta\in(0,1)$
larger than $B/e_{n,\infty}$. The limits just proved imply that
$L_r/e_{n,r}\le\theta$ for all sufficiently large $r$. Therefore
\[
 e_{n,r}(1-\theta^r)^{1/r}
 \le e_{n,r}\left(1-\left(\frac{L_r}{e_{n,r}}\right)^r\right)^{1/r}
 =\widehat e_{n,r}\le e_{n,r}.
\]
Since $r^{-1}\log(1-\theta^r)\to0$, the left-hand multiplier tends to one.
Both outside quantities tend to $e_{n,\infty}$, proving
\eqref{eq:hatted-gap-limit}.
\end{proof}

\begin{example}\label{ex:order-endpoint}
Let $P=\tfrac12\delta_0+\tfrac12\delta_{10}$, $S=\{0,1\}$, and $n=1$.
For $r\ge1$, the center $1$ is optimal, since $1+9^r\le10^r$. Hence
\[
 e_{1,r}^{r}=\frac{1+9^r}{2},\qquad
 L_r^r=\frac{9^r}{2},\qquad
 \widehat e_{1,r}=2^{-1/r}.
\]
In particular,
\[
 \widetilde e_{1,1}=\frac12,
 \qquad
 \widetilde e_{1,2}=\sqrt{41}-\sqrt{81/2}<\frac12.
\]
The last inequality follows by rationalizing the difference: its numerator
is $1/2$ and its denominator is $\sqrt{41}+\sqrt{81/2}>1$.
Also $e_{1,\infty}=B=9$, whereas $\widehat e_{1,r}\to1$.

\end{example}

\section{Upper Asymptotics for Constrained Quantization Dimensions}
\label{section:upper_quantization_dimension_estimate}
\subsection{Comparison of admissible centers}

The projection set and $E_\mu$ lie in $S$, so their codebooks are admissible
for constrained quantization. The next comparison moves unrestricted
centers into $E_\mu$ while controlling the increase in their error.

\begin{lemma}\label{lem:factor-two}
If $E_\mu=\supp\mu$ is compact, then for $r>0$,
\begin{equation}\label{eq:factor-two}
 q_{n,r}(\mu)\le q_{n,r}(\mu;E_\mu)\le2q_{n,r}(\mu).
\end{equation}
For any nonempty compact set $A$,
\begin{equation}\label{eq:factor-two-cover}
 c_n(A)\le c_n(A;A)\le2c_n(A).
\end{equation}
\end{lemma}
\begin{proof}
The first inequalities follow because restricting the codebooks can only
increase an infimum. For the reverse comparison, take an arbitrary
$\alpha\in\Acal_n(\R^D)$ and, for each $a\in\alpha$, choose $a'\in E_\mu$ with
$|a-a'|=\rho(a,E_\mu)$. Let $\alpha'$ be the set of these chosen points.
It is nonempty and has at most $n$ elements. For $y\in E_\mu$,
\[
 |y-a'|\le |y-a|+\rho(a,E_\mu)\le2|y-a|.
\]
Choose $a$ nearest to $y$ among the original centers. Then
$\rho(y,\alpha')\le2\rho(y,\alpha)$. Integration, then taking the $r$th root, gives
\[
 q_{n,r}(\mu;E_\mu)
 \le\left(\int\rho(y,\alpha')^r\,d\mu(y)\right)^{1/r}
 \le2\left(\int\rho(y,\alpha)^r\,d\mu(y)\right)^{1/r}.
\]
This holds for every original $\alpha$, so taking their infimum gives
\eqref{eq:factor-two}. For \eqref{eq:factor-two-cover}, make the same
replacement in $A$ and take the supremum over $y\in A$ instead of an integral.
The constant $2$ in the upper inequality of
\eqref{eq:factor-two-cover} is sharp. Indeed, for $A=\{-1,1\}$,
the unrestricted center $0$ gives $c_1(A)=1$, while either center
in $A$ gives $c_1(A;A)=2$. Thus no smaller universal constant can
replace $2$ in that inequality.
\end{proof}

\subsection{Proof of the projection error estimates}

\begin{proof}[Proof of the finite-order estimates in Theorem~\ref{thm:upper_asymptotics}]
Fix $\beta\in\Acal_n(E_\mu)$. For every $x\in K$, choose a point of $\beta$
nearest to $T(x)$. Since $|x-T(x)|=b(x)$, the triangle inequality gives
\begin{equation}\label{eq:pointwise-upper}
 b(x)\le\rho(x,\beta)
 \le b(x)+\rho(T(x),\beta).
\end{equation}
For $r\ge1$, apply Minkowski to the nonnegative functions $b$ and
$g_\beta(x)=\rho(T(x),\beta)$. The pushforward identity gives
\begin{align*}
 F_r(\beta)
 &\le\|b+g_\beta\|_{L^r(P)}\\
 &\le\|b\|_{L^r(P)}+\|g_\beta\|_{L^r(P)}\\
 &=L_r+\left(\int_{E_\mu}\rho(y,\beta)^r\,d\mu(y)\right)^{1/r}.
\end{align*}
Since $\beta\in\Acal_n(E_\mu)\subset\Acal_n(S)$, the definition
$e_{n,r}(P;S)=\inf_{\alpha\in\Acal_n(S)}F_r(\alpha)$ gives
$e_{n,r}\le F_r(\beta)$. Subtracting
$L_r$ and taking the infimum over $\beta\in\Acal_n(E_\mu)$ proves the first
inequality in \eqref{eq:additive-upper}. Its second inequality is
Lemma~\ref{lem:factor-two}.

To estimate $\Delta_{n,r}=e_{n,r}^r-L_r^r$, observe that $T(x)$ and all centers of $\beta$
belong to $E$, so $g_\beta(x)\le\diam E$ and $b(x)+g_\beta(x)<M_0$.
For $r\ge1$, we have
\begin{align*}
 0\le\rho(x,\beta)^r-b(x)^r
 &\le(b(x)+g_\beta(x))^r-b(x)^r\\
 &=r\int_{b(x)}^{b(x)+g_\beta(x)}t^{r-1}\,dt\\
 &\le rM_0^{r-1}g_\beta(x).
\end{align*}
By admissibility, integration, and \eqref{eq:pushforward-integral},
\begin{align*}
 \Delta_{n,r}
 &\le F_r(\beta)^r-L_r^r\\
 &\le rM_0^{r-1}\int_K\rho(T(x),\beta)\,dP(x)\\
 &=rM_0^{r-1}\int_{E_\mu}\rho(y,\beta)\,d\mu(y).
\end{align*}
Taking the infimum over $\beta\in\Acal_n(E_\mu)$ gives
\[
\Delta_{n,r}\le rM_0^{r-1}q_{n,1}(\mu;E_\mu)
\le 2rM_0^{r-1}q_{n,1}(\mu),
\]
where the last step is \eqref{eq:factor-two} at order $1$.
This proves \eqref{eq:power-upper}.

For $0<r<1$, use $(u+v)^r\le u^r+v^r$ for $u,v\ge0$ in
\eqref{eq:pointwise-upper}. This gives
\[
 0\le\rho(x,\beta)^r-b(x)^r\le g_\beta(x)^r.
\]
For every $\beta\in\Acal_n(E_\mu)$, integration gives
\[
\Delta_{n,r}\le F_r(\beta)^r-L_r^r
\le\int_K\rho(T(x),\beta)^r\,dP(x)
=\int_{E_\mu}\rho(y,\beta)^r\,d\mu(y).
\]
Taking the infimum over these sets $\beta$ and using
\eqref{eq:factor-two}, raised to the positive power $r$, yields
\[
\Delta_{n,r}\le q_{n,r}(\mu;E_\mu)^r\le2^rq_{n,r}(\mu)^r,
\]
which is \eqref{eq:subunit-upper}.
\end{proof}

\begin{proof}[Proof of the infinite-order estimate in Theorem~\ref{thm:upper_asymptotics}]
Fix $\beta\in\Acal_n(E)$. For $x\in K$, take $p\in\pi_S(x)$ and then a
point of $\beta$ nearest to $p$. We obtain
\[
 \rho(x,\beta)\le b(x)+\rho(p,\beta)
 \le B+\sup_{y\in E}\rho(y,\beta).
\]
Take the supremum over $x\in K$. Since this codebook is also admissible in
$S$, it follows that
\[
 e_{n,\infty}-B\le\sup_{y\in E}\rho(y,\beta).
\]
Taking the infimum over $\beta\in\Acal_n(E)$ gives the middle inequality
of \eqref{eq:infty-upper}; Lemma~\ref{lem:factor-two} gives the last one.
The lower inequality is the baseline bound of Lemma~\ref{prop:baseline}.
\end{proof}

\subsection{Box dimension bounds}
For a bounded set $A$, let $N_t(A)$ be the least number of closed balls of
radius $t$ needed to cover $A$, with unrestricted centers. Thus
\[
\Blo A=\liminf_{t\downarrow0}\frac{\log N_t(A)}{-\log t},\qquad
\Bup A=\limsup_{t\downarrow0}\frac{\log N_t(A)}{-\log t}.
\]
Every cover of a set also covers each of its subsets, so
\begin{equation}\label{eq:cover-inclusion}
A\subset A'\quad\Longrightarrow\quad N_t(A)\le N_t(A')
\qquad(t>0).
\end{equation}

\begin{lemma}\label{lem:box-support}
For compact $E_\mu=\supp\mu$, the full-mass box dimensions satisfy
\begin{equation}\label{eq:box-support}
 \Blo^*\mu:=\inf_{\substack{A\subset E_\mu\\\mu(A)=1}}\Blo A=\Blo E_\mu,
 \qquad
 \Bup^*\mu:=\inf_{\substack{A\subset E_\mu\\\mu(A)=1}}\Bup A=\Bup E_\mu,
\end{equation}
where the infima range over Borel sets. Also, for every $r>0$,
\begin{equation}\label{eq:classical-box-upper}
 \Dlo_r(\mu)\le\Blo E_\mu,\qquad \Dup_r(\mu)\le\Bup E_\mu.
\end{equation}
The same inequalities hold for the dimensions of $c_n(E_\mu)$ and $c_n(E_\mu;E_\mu)$.
\end{lemma}
\begin{proof}
Let $A\subset E_\mu$ be Borel with $\mu(A)=1$. Every ball centered at a point
of $E_\mu$ meets $A$, so $\overline A=E_\mu$. A finite union of closed balls
covering $A$ also covers $\overline A$. Hence, for every $t>0$,
\[
N_t(A)=N_t(E_\mu).
\]
Taking the lower and upper limits, and then the infima over $A$, proves
\eqref{eq:box-support}.

For the covering argument, suppose first that $E_\mu$ is finite.
For $n\ge|E_\mu|$, the codebook $E_\mu$ itself gives
\[
q_{n,r}(\mu)=c_n(E_\mu)=c_n(E_\mu;E_\mu)=0.
\]
Their lower and upper logarithmic dimensions are therefore zero by
definition. Also $N_t(E_\mu)\le|E_\mu|$ for every $t>0$, so
$\Blo E_\mu=\Bup E_\mu=0$.
Suppose $E_\mu$ is infinite. For every $t>0$, a cover by $N_t(E_\mu)$ balls yields
an unrestricted codebook with maximum error at most $t$. Hence
\begin{equation}\label{eq:cover-error}
 c_{N_t(E_\mu)}(E_\mu)\le t,\qquad
 q_{N_t(E_\mu),r}(\mu)\le t.
\end{equation}
Let $s>\Bup E_\mu$. For sufficiently small $t$, $N_t(E_\mu)\le t^{-s}$.
Set $t_n=(2/n)^{1/s}$. For large $n$, $N_{t_n}(E_\mu)\le n/2\le n$,
so \eqref{eq:cover-error} and monotonicity give
\[
 c_n(E_\mu)\le t_n,\qquad q_{n,r}(\mu)\le t_n.
\]
Lemma~\ref{lem:log-transfer}(iii) bounds both upper dimensions by $s$;
let $s\downarrow\Bup E_\mu$.
For $s>\Blo E_\mu$, choose $t_j\downarrow0$ with
$n_j=N_{t_j}(E_\mu)\le t_j^{-s}$. Then $t_j\le n_j^{-1/s}$.
We claim that $n_j\to\infty$. Otherwise, along a subsequence there would
be an integer $N$ and centers $z_{j,1},\ldots,z_{j,N}$ with
\[
 E_\mu\subset\bigcup_{i=1}^N\overline B(z_{j,i},t_j).
\]
Discard balls disjoint from $E_\mu$, and repeat a remaining center if
necessary to keep $N$ indices. Choose
$y_{j,i}\in E_\mu\cap\overline B(z_{j,i},t_j)$. Then
\[
 E_\mu\subset\bigcup_{i=1}^N\overline B(y_{j,i},2t_j).
\]
Compactness of $E_\mu^N$ gives a further subsequence with
$y_{j,i}\to y_i\in E_\mu$ for every $i$. For each $x\in E_\mu$,
\[
 \min_{1\le i\le N}|x-y_i|
 =\lim_{j\to\infty}\min_{1\le i\le N}|x-y_{j,i}|
 \le\lim_{j\to\infty}2t_j=0.
\]
Thus $E_\mu\subset\{y_1,\ldots,y_N\}$, contradicting its infinitude.
Consequently
\[
 c_{n_j}(E_\mu)\le t_j\le n_j^{-1/s},\qquad
 q_{n_j,r}(\mu)\le t_j\le n_j^{-1/s}.
\]
Lemma~\ref{lem:log-transfer}(iv) bounds both lower dimensions by $s$;
let $s\downarrow\Blo E_\mu$.
Finally, \eqref{eq:factor-two-cover} and Lemma~\ref{lem:log-transfer}(ii)
give the same assertions for $c_n(E_\mu;E_\mu)$.
\end{proof}

\begin{proof}[Proof of the dimension estimates in Theorem~\ref{thm:upper_asymptotics}]
For $r\ge1$, \eqref{eq:additive-upper} and Lemma~\ref{lem:log-transfer}(i), with
$(a_n,b_n,\theta)=(\widetilde e_{n,r},q_{n,r}(\mu),1)$, give
\[
 \ldimqtilde r{P;S}\le\Dlo_r(\mu),\qquad
 \udimqtilde r{P;S}\le\Dup_r(\mu).
\]
Lemma~\ref{lem:box-support} bounds these by $\Blo E_\mu$ and $\Bup E_\mu$.
Since $E_\mu\subset E$, $N_t(E_\mu)\le N_t(E)$, hence
$\Blo E_\mu\le\Blo E$ and $\Bup E_\mu\le\Bup E$.
Moreover, \eqref{eq:power-upper} gives
\[
 \widehat e_{n,r}\le(2rM_0^{r-1})^{1/r}q_{n,1}(\mu)^{1/r}.
\]
Lemma~\ref{lem:log-transfer}(i), with $\theta=1/r$, gives
$\ldimqhat r{P;S}\le r\Dlo_1(\mu)$ and $\udimqhat r{P;S}\le r\Dup_1(\mu)$.
Apply \eqref{eq:classical-box-upper} at order one and divide by $r$.
\end{proof}

\begin{corollary}\label{prop:estimate_upper_lower_constrained_quantization_dim}
For $0<r<1$ and $L_r>0$,
\begin{align}
r\ldimqtilde{r}{P;S}&\le\Blo^*\mu,&r\udimqtilde{r}{P;S}&\le\Bup^*\mu,
\label{eq:corrected-subunit-tilde}\\
\ldimqhat{r}{P;S}&\le\Blo^*\mu,&\udimqhat{r}{P;S}&\le\Bup^*\mu.
\label{eq:corrected-subunit-hat}
\end{align}
If $L_r=0$, then for every $r>0$,
\[
 \begin{aligned}
 \ldimqtilde r{P;S}=\ldimqhat r{P;S}&\le\Blo^*\mu,\\
 \udimqtilde r{P;S}=\udimqhat r{P;S}&\le\Bup^*\mu.
 \end{aligned}
\] Also,
\begin{equation}\label{eq:infty-box}
\ldimqtilde{\infty}{P;S}\le\Blo E,\qquad\udimqtilde{\infty}{P;S}\le\Bup E.
\end{equation}
\end{corollary}

\begin{proof}
For $0<r<1$, \eqref{eq:subunit-upper} gives
$\widehat e_{n,r}\le2q_{n,r}(\mu)$. If $L_r>0$, \eqref{eq:conversion} gives
\[
 \widetilde e_{n,r}\le m_r^{-1}\Delta_{n,r}
 \le2^rm_r^{-1}q_{n,r}(\mu)^r.
\]
Apply Lemma~\ref{lem:log-transfer}(i) with exponents $1$ and $r$, respectively, and then Lemma~\ref{lem:box-support}. This proves \eqref{eq:corrected-subunit-hat} and \eqref{eq:corrected-subunit-tilde}.
If $L_r=0$, then $K\subset S$, $T$ is the identity on $K$, and $\mu=P$. Thus
\[
 q_{n,r}(P)\le e_{n,r}(P;S)\le q_{n,r}(P;K)\le2q_{n,r}(P).
\]
Both excesses equal $e_{n,r}$; Lemma~\ref{lem:log-transfer}(ii) and Lemma~\ref{lem:box-support} give their box-dimension bounds.
At infinite order, \eqref{eq:infty-upper} gives $\widetilde e_{n,\infty}\le2c_n(E)$. Apply Lemma~\ref{lem:log-transfer}(i) with exponent one and the covering-dimension bounds in Lemma~\ref{lem:box-support}.
\end{proof}

\subsection{Upper rates from lower Ahlfors estimates}

\begin{proposition}\label{prop:packing}\label{lemma:upperbound_limsup_n_e_n_infty_D_prime}
Let $A\subset\R^D$ be nonempty and compact. Suppose $\eta$ is a finite Borel
measure and
\begin{equation}\label{eq:lower-mass}
 \eta(B(a,t))\ge c t^d\qquad(a\in A,\ 0<t<t_0),
\end{equation}
where $c,t_0>0$ and $d>0$. Then
\begin{equation}\label{eq:packing-rate}
 \limsup_{n\to\infty}n c_n(A;A)^d
 \le\frac{2^d\eta(\R^D)}c<\infty.
\end{equation}
There is a finite constant $C_A$, independent of $n$, with
$c_n(A;A)\le C_A n^{-1/d}$ for every $n\ge1$.
If \eqref{eq:lower-mass} holds with $d=0$, then $A$ is finite and the radii
are eventually zero.
\end{proposition}
\begin{proof}
Put $R_n=c_n(A;A)$. For every $\varepsilon>0$, compactness gives
points $a_1,\ldots,a_{N_\varepsilon}\in A$ such that
\[
A\subset\bigcup_{i=1}^{N_\varepsilon}B(a_i,\varepsilon).
\]
For $n\ge N_\varepsilon$, the set
$\{a_1,\ldots,a_{N_\varepsilon}\}$ belongs to $\Acal_n(A)$, whence
\[
0\le R_n\le\max_{x\in A}\min_{1\le i\le N_\varepsilon}|x-a_i|
\le\varepsilon.
\]
Thus $R_n\to0$ as $n\to\infty$. If $R_n=0$ for some $n$, then
existence of a minimizing codebook on the compact set $A$ gives $|A|\le n$,
and all subsequent errors vanish. We may therefore consider $R_n>0$.
Choose a subset $\gamma_n\subset A$ of largest cardinality subject to
\[
 |a-a'|\ge R_n\quad(a,a'\in\gamma_n,\ a\ne a').
\]
Such a largest finite cardinality exists because a finite cover of the
compact set by sets of diameter less than $R_n$ bounds the cardinality of
every separated set. Write $m_n=|\gamma_n|$.

We claim that $m_n>n$. If $m_n\le n$, admissibility gives
$\sup_{x\in A}\rho(x,\gamma_n)\ge R_n$. The supremum is attained at some
$z\in A$ because the distance function is continuous. Since $R_n>0$,
$z\notin\gamma_n$, and $|z-a|\ge R_n$ for each $a\in\gamma_n$. Adding $z$
would contradict maximal cardinality.

The open balls $B(a,R_n/2)$ for $a\in\gamma_n$ are disjoint. For large $n$,
$R_n/2<t_0$, so
\begin{align*}
 \eta(\R^D)
 &\ge\eta\left(\bigcup_{a\in\gamma_n}B(a,R_n/2)\right)\\
 &=\sum_{a\in\gamma_n}\eta(B(a,R_n/2))\\
 &\ge m_n c(R_n/2)^d>n c(R_n/2)^d.
\end{align*}
This proves \eqref{eq:packing-rate}.  Choose $n_0$ so that $R_n/2<t_0$ for $n\ge n_0$.
With the maximum over an empty set interpreted as zero, put
\[
C_A=\max\left\{
\left(\frac{2^d\eta(\R^D)}c\right)^{1/d},
\max_{1\le n<n_0}n^{1/d}R_n\right\}.
\]
This constant is finite and independent of $n$, and the preceding
estimate gives $R_n\le C_A n^{-1/d}$ for every integer $n\ge1$.

If $d=0$, continuity from above for the finite measure gives
$\eta(\{a\})=\lim_{t\downarrow0}\eta(B(a,t))\ge c$ for each $a\in A$.
There can be at most $\eta(\R^D)/c$ such points. Thus $A$ is finite.
\end{proof}

\begin{corollary}\label{cor:regular-upper}\label{corollary:upperbound_prob_compactly_supported}
Suppose $E=\pi_S(K)$ is Ahlfors regular of dimension $d>0$, with respect to $\mathcal H^d|_E$. Then, for each fixed finite $r>0$, the following
upper estimates hold with constants independent of $n$:
\begin{center}
\begin{tabular}{lll}
\toprule
Conditions&$\widetilde e_{n,r}=e_{n,r}-L_r$&$\widehat e_{n,r}=(e_{n,r}^r-L_r^r)^{1/r}$\\\midrule
$r\ge1$, $L_r>0$&$\widetilde e_{n,r}\le C_rn^{-1/d}$&$\widehat e_{n,r}\le C_r'n^{-1/(rd)}$\\
$0<r<1$, $L_r>0$&$\widetilde e_{n,r}\le C_rn^{-r/d}$&$\widehat e_{n,r}\le C_r'n^{-1/d}$\\
$r>0$, $L_r=0$&$\widetilde e_{n,r}\le C_rn^{-1/d}$&$\widehat e_{n,r}=\widetilde e_{n,r}$\\
\bottomrule
\end{tabular}
\end{center}
Also $\widetilde e_{n,\infty}\le C_\infty n^{-1/d}$.
Thus, at positive baseline,
\begin{align}
 \udimqtilde{r}{P;S}&\le d,&\udimqhat{r}{P;S}&\le rd&& (r\ge1),\label{eq:regular-upper-dim1}\\
 \udimqtilde{r}{P;S}&\le d/r,&\udimqhat{r}{P;S}&\le d&& (0<r<1),\label{eq:regular-upper-dim2}
\end{align}
and $\udimqtilde{\infty}{P;S}\le d$. The corresponding lower dimensions satisfy the
same bounds. No regularity of $\mu$ is imposed in this corollary.
\end{corollary}
\begin{proof}
Set $\eta=\mathcal H^d|_E$. By the assumed Ahlfors regularity,
$0<\eta(E)<\infty$ and there are $c_E,t_E>0$ such that
\[
\eta(B(y,t))=\mathcal H^d(E\cap B(y,t))\ge c_Et^d
\qquad(y\in E,\ 0<t<t_E).
\]
Apply Proposition~\ref{prop:packing} with $A=E$ and this measure
$\eta$ to obtain $c_n(E;E)\le C_E n^{-1/d}$. Choose an optimal covering codebook
$\beta_n\subset E$; such a codebook exists by compactness. For every $s>0$,
\[
 \int_E\rho(y,\beta_n)^s\,d\mu(y)
 \le\sup_{y\in E}\rho(y,\beta_n)^s
 =c_n(E;E)^s.
\]
The pointwise estimate \eqref{eq:pointwise-upper} applies to $\beta_n$.

For $r\ge1$ this gives
\[
 \widetilde e_{n,r}\le C_E n^{-1/d},\qquad
 \Delta_{n,r}\le rM_0^{r-1}C_E n^{-1/d},
\]
and taking the $r$th root of the second inequality gives the first row.
For $0<r<1$, it gives
\[
 \Delta_{n,r}\le C_E^r n^{-r/d},\qquad
 \widehat e_{n,r}\le C_E n^{-1/d}.
\]
If $L_r>0$, \eqref{eq:conversion} gives
\[
 \widetilde e_{n,r}=e_{n,r}-L_r
 \le m_r^{-1}\Delta_{n,r}
 \le m_r^{-1}C_E^r n^{-r/d}.
\] If $L_r=0$, then $E=K$ and
$\widetilde e_{n,r}=\widehat e_{n,r}=e_{n,r}(P;S)$. The covering codebook directly gives
$e_{n,r}\le c_n(K;K)\le C_E n^{-1/d}$ for every $r>0$.
Theorem~\ref{thm:upper_asymptotics} gives the infinite-order bound.

Apply Lemma~\ref{lem:log-transfer}(iii) to the error estimates just proved. At positive baseline, the exponents for $(\widetilde e_{n,r},\widehat e_{n,r})$ are
\[
 (1/d,1/(rd))\quad(r\ge1),\qquad
 (r/d,1/d)\quad(0<r<1).
\]
Their reciprocal exponents give \eqref{eq:regular-upper-dim1} and \eqref{eq:regular-upper-dim2}. The exponent $1/d$ gives $\udimqtilde\infty{P;S}\le d$ and both finite-order bounds at zero baseline. Each lower dimension is at most its upper dimension.
\end{proof}

\section{Lower Asymptotics for Constrained Quantization Dimensions}
\label{sec:lower}
\label{section:lower_quantization_dimension_estimate}
\label{sec:small-ball-link}

Use the compact set $C$ and the constant $H$ from
\eqref{eq:common-compact}. Proposition~\ref{prop:support} gives
$E_\mu\subset E$, and the bound in the proof of
Lemma~\ref{lem:projection-topology} gives $E\subset C$. The localization in
Lemma~\ref{prop:existence} gives, for every $r>0$,
\[
 \Delta_{n,r}
 =\inf_{\alpha\in\Acal_n(C)}
   \int_K\bigl(\rho(x,\alpha)^r-b(x)^r\bigr)\,dP(x).
\]

For a finite nonnegative Borel measure $\nu$ define
\begin{equation}\label{eq:V}
 \Vcal_{n,\beta}(\nu;A)
 =\inf_{\alpha\in\Acal_n(A)}\int\dist(y,\alpha)^\beta\,d\nu(y),
 \qquad \Vcal_{n,\beta}(\nu)=\Vcal_{n,\beta}(\nu;\R^D).
\end{equation}
For a probability $\eta$, write
$q_{n,\beta}(\eta;A)=\Vcal_{n,\beta}(\eta;A)^{1/\beta}$ and omit
$A$ when $A=\R^D$. If $W=\nu(\R^D)>0$, then
$\Vcal_{n,\beta}(\nu;A)=Wq_{n,\beta}(\nu/W;A)^\beta$;
for the zero measure the cost is zero. 

For a compactly supported probability $\eta$ and a positive
$v\in L^1(\eta)$, the $A_p(\eta)$ constant is
\begin{equation}\label{eq:Ap}
 [v]_{A_p(\eta)}
 =\sup_B\left(\fint_Bv\,d\eta\right)
       \left(\fint_Bv^{-1/(p-1)}\,d\eta\right)^{p-1},
 \qquad 1<p<\infty.
\end{equation}
Here $\fint_B=\eta(B)^{-1}\int_B$, and the supremum is over all
positive-radius balls centered in $\supp\eta$. We require $v>0$
$\eta$-a.e. and write $v\in A_p(\eta)$ if $[v]_{A_p(\eta)}<\infty$.
For each such ball,
\[
 0<\int_Bv\,d\eta<\infty,
\]
and the definition gives
\[
 \int_Bv^{-1/(p-1)}\,d\eta
 \le\eta(B)\left(
 \frac{[v]_{A_p(\eta)}\eta(B)}{\int_Bv\,d\eta}
 \right)^{1/(p-1)}<\infty.
\]

\begin{theorem}[Weighted projection lower bound]\label{thm:master}
Fix $r,\beta>0$. Let $F\subset K$ be Borel and let
$w:F\to[0,\infty)$ be Borel, with $\int_Fw\,dP<\infty$. Suppose
\begin{equation}\label{eq:growth}
 |x-a|^r-b(x)^r\ge w(x)|a-T(x)|^\beta
 \qquad(x\in F,\ a\in C).
\end{equation}
For $\nu_F=T_*(wP|_F)$ and every $n\ge1$,
\begin{equation}\label{eq:transfer}
 \Delta_{n,r}\ge\Vcal_{n,\beta}(\nu_F;C)
                       \ge\Vcal_{n,\beta}(\nu_F).
\end{equation}
If moreover $d\nu_F=\wh\,d\eta$ for a compactly supported probability
$\eta$ and $\wh\in A_p(\eta)$, put $W=\nu_F(\R^D)$. Then
\begin{equation}\label{eq:Ap-transfer}
 \Delta_{n,r}\ge
 \frac{W}{[\wh]_{A_p(\eta)}}q_{n,\beta/p}(\eta)^\beta.
\end{equation}
In particular, if $d,M>0$ and
\begin{equation}\label{eq:ball}
 \eta(B(y,t))\le Mt^d\qquad(y\in\R^D,\ t>0),
\end{equation}
then
\begin{equation}\label{eq:Ap-rate}
 \Delta_{n,r}\ge
 \frac{W}{[\wh]_{A_p(\eta)}}
 \left(\frac{d}{d+\beta/p}\right)^p
 M^{-\beta/d}n^{-\beta/d}.
\end{equation}
\end{theorem}

\begin{proof}
For $\alpha\in\Acal_n(C)$ and $x\in K$, every $a\in\alpha$
belongs to $S$, so $|x-a|\ge b(x)$ and
$\dist(x,\alpha)^r-b(x)^r\ge0$. For $x\in F$, \eqref{eq:growth}
gives
\[
 \dist(x,\alpha)^r-b(x)^r
 =\min_{a\in\alpha}\bigl(|x-a|^r-b(x)^r\bigr)
 \ge w(x)\min_{a\in\alpha}|a-T(x)|^\beta
 =w(x)\dist(T(x),\alpha)^\beta.
\]
Apply \eqref{eq:pushforward-finite} to the finite Borel measure
$\lambda=wP|_F$. Since $\nu_F=T_*\lambda$,
\begin{equation}\label{eq:weighted-integration}
 \int_{\R^D}\varphi(y)\,d\nu_F(y)
 =\int_F\varphi(T(x))w(x)\,dP(x)
 \quad(\varphi:\R^D\to[0,\infty]\text{ Borel}).
\end{equation}
In particular $\nu_F(\R^D)=\int_Fw\,dP<\infty$. Integration of
the pointwise inequalities and then \eqref{eq:weighted-integration}
with $\varphi(y)=\dist(y,\alpha)^\beta$ give
\begin{align*}
 F_r(\alpha)^r-L_r^r
 &=\int_F\bigl(\dist(x,\alpha)^r-b(x)^r\bigr)\,dP(x)
   +\int_{K\setminus F}\bigl(\dist(x,\alpha)^r-b(x)^r\bigr)\,dP(x)\\
 &\ge\int_F\bigl(\dist(x,\alpha)^r-b(x)^r\bigr)\,dP(x)\\
 &\ge\int_Fw(x)\dist(T(x),\alpha)^\beta\,dP(x)\\
 &=\int_{\R^D}\dist(y,\alpha)^\beta\,d\nu_F(y).
\end{align*}

By the localization in Lemma~\ref{prop:existence},
\[
 \begin{aligned}
 \Delta_{n,r}
 &=\inf_{\alpha\in\Acal_n(C)}\bigl(F_r(\alpha)^r-L_r^r\bigr)\\
 &\ge\inf_{\alpha\in\Acal_n(C)}
       \int\dist(y,\alpha)^\beta\,d\nu_F(y)\\
 &=\Vcal_{n,\beta}(\nu_F;C)
 \ge\Vcal_{n,\beta}(\nu_F).
 \end{aligned}
\]
The last inequality follows from $\Acal_n(C)\subset\Acal_n(\R^D)$.

Now suppose that $d\nu_F=\wh\,d\eta$ and $\wh\in A_p(\eta)$. A ball
containing $\supp\eta$ has $\eta$-mass one. On that ball the definition
gives, with $I=\int\wh^{-1/(p-1)}d\eta$,
\begin{equation}\label{eq:global-Ap}
 WI^{p-1}\le[\wh]_{A_p(\eta)}.
\end{equation}
In particular $0<W<\infty$ and $I<\infty$. For any finite codebook
write $D_\alpha(y)=\dist(y,\alpha)$. H\"older's inequality gives
\[
 \int D_\alpha^{\beta/p}\,d\eta
 =\int(\wh D_\alpha^\beta)^{1/p}\wh^{-1/p}\,d\eta
 \le\left(\int\wh D_\alpha^\beta\,d\eta\right)^{1/p}
       I^{(p-1)/p}.
\]
Consequently
\[
 \int D_\alpha^\beta\,d\nu_F
 \ge I^{-(p-1)}\left(\int D_\alpha^{\beta/p}\,d\eta\right)^p
 \ge\frac{W}{[\wh]_{A_p(\eta)}}
       \left(\int D_\alpha^{\beta/p}\,d\eta\right)^p.
\]
Taking the infimum over $\alpha\in\Acal_n(\R^D)$ gives
\[
 \Vcal_{n,\beta}(\nu_F)
 \ge\frac{W}{[\wh]_{A_p(\eta)}}
       \Vcal_{n,\beta/p}(\eta)^p
 =\frac{W}{[\wh]_{A_p(\eta)}}q_{n,\beta/p}(\eta)^\beta.
\]
Together with \eqref{eq:transfer}, this proves
\eqref{eq:Ap-transfer}.

Under \eqref{eq:ball}, a codebook of at most $n$ points satisfies
$\eta\{D_\alpha<t\}\le nMt^d$. For $s>0$, Tonelli's theorem and
$h=(nM)^{-1/d}$ therefore imply
\begin{equation}\label{eq:ball-moment}
 \begin{aligned}
 \int D_\alpha^s\,d\eta
 &=\int_0^\infty st^{s-1}\eta\{D_\alpha\ge t\}\,dt\\
 &\ge\int_0^h st^{s-1}(1-nMt^d)\,dt
 =\frac{d}{d+s}M^{-s/d}n^{-s/d}.
 \end{aligned}
\end{equation}
The right side does not depend on
$\alpha\in\Acal_n(\R^D)$, so its infimum gives the same lower
bound for $\Vcal_{n,s}(\eta)$. With $s=\beta/p$, all quantities
are nonnegative and
\begin{align*}
 q_{n,\beta/p}(\eta)^\beta
 &=\Vcal_{n,\beta/p}(\eta)^{\beta/(\beta/p)}
   =\Vcal_{n,\beta/p}(\eta)^p\\
 &\ge\left(\frac{d}{d+\beta/p}
             M^{-\beta/(pd)}n^{-\beta/(pd)}\right)^p\\
 &=\left(\frac{d}{d+\beta/p}\right)^p
             M^{-\beta/d}n^{-\beta/d}.
\end{align*}
Substitution in \eqref{eq:Ap-transfer} proves \eqref{eq:Ap-rate}.
The inequality \eqref{eq:ball-moment} is the ordinary power-distortion
case of \cite[Theorem~4.1 and Appendix~D]{riegler2023lossy}.
\end{proof}

\begin{corollary}[Lower estimates from positive-mass growth]\label{thm:lower_asymptotics}
Suppose \eqref{eq:upper-mass} and \eqref{cond:Gbeta} hold. Put
\[
c_*=2^{-1-\beta/d}M^{-\beta/d},\qquad
A_r=rb_0^{r-1}c\,c_*m^{1+\beta/d}.
\]
For every fixed $1\le r<\infty$ and every $n\ge1$,
\begin{align}
\Delta_{n,r}&\ge A_r n^{-\beta/d},\label{eq:lower-power}\\
\widetilde e_{n,r}&\ge \frac{A_r}{rH^{r-1}}n^{-\beta/d},&
\widehat e_{n,r}&\ge A_r^{1/r}n^{-\beta/(rd)}.
\label{eq:lower-excess}
\end{align}
In particular,
\begin{equation}\label{eq:lower-dimension}
\ldimqtilde{r}{P;S}\ge\frac d\beta,
\qquad \ldimqhat{r}{P;S}\ge\frac{rd}\beta.
\end{equation}
\end{corollary}
\begin{proof}
Apply Lemma~\ref{lem:uniform-growth} with $A=C$.
For $t=\beta/d$,
\[
 \frac1{1+t}=\int_0^1u^t\,du
 \ge\int_{1/2}^1 2^{-t}\,du=2^{-1-t}.
\]
Hence the constant $\Lambda_{r,F}$ in that lemma satisfies
$\Lambda_{r,F}\ge A_r$. Its three error bounds and two dimension bounds
therefore give \eqref{eq:lower-power}--\eqref{eq:lower-dimension}.
\end{proof}

\begin{proof}[Proof of Theorem~\ref{thm:weighted_lower_main}]
Condition~\eqref{eq:weighted-main-growth} is the pointwise hypothesis
\eqref{eq:growth} of Theorem~\ref{thm:master}, with the same
$r,\beta,F,w$. The measure $\nu_F=T_*(wP|_F)$ is finite, with
\[
 W=\nu_F(\R^D)=\int_Fw\,dP=\int\wh\,d\eta<\infty.
\]
Since $\wh>0$ $\eta$-a.e. and $\eta$ is a probability, $W>0$.
Every ball centered in $\supp\eta$ has positive $\eta$-mass, and
both averages in \eqref{eq:Ap} are strictly positive. The assumption
$\wh\in A_p(\eta)$ therefore gives
$0<[\wh]_{A_p(\eta)}<\infty$, hence $0<\kappa<\infty$.
The comparison \eqref{eq:transfer}, together with the H\"older
estimate for $\Vcal_{n,\beta}(\nu_F)$ established in the proof of
Theorem~\ref{thm:master}, gives
\[
 \Delta_{n,r}\ge\Vcal_{n,\beta}(\nu_F;C)
 \ge\Vcal_{n,\beta}(\nu_F)
 \ge\frac{W}{[\wh]_{A_p(\eta)}}q_{n,\beta/p}(\eta)^\beta
 =\kappa q_{n,s}(\eta)^\beta.
\]

To obtain the other two error bounds, fix $y_0\in E$. For every
$x\in K$, the selected point $T(x)\in E$ satisfies
\[
 |x-y_0|\le |x-T(x)|+|T(x)-y_0|
 \le b(x)+\diam E\le B+\diam E<M_0.
\]
The singleton $\{y_0\}$ is an admissible codebook for every $n$.
Thus
\[
 0\le L_r\le e_{n,r}\le F_r(\{y_0\})
 \le B+\diam E<M_0.
\]
Since $r\ge1$, one has $u^{r-1}\le M_0^{r-1}$ throughout
$[L_r,e_{n,r}]$, and so
\[
 \Delta_{n,r}
 =e_{n,r}^r-L_r^r
 =r\int_{L_r}^{e_{n,r}}u^{r-1}\,du
 \le rM_0^{r-1}(e_{n,r}-L_r)
 =rM_0^{r-1}\widetilde e_{n,r}.
\]
Together with \eqref{eq:weighted-main-power}, this gives
\[
 \widetilde e_{n,r}
 \ge\frac{\Delta_{n,r}}{rM_0^{r-1}}
 \ge\frac{\kappa}{rM_0^{r-1}}q_{n,s}(\eta)^\beta,
\]
and
\[
 \widehat e_{n,r}=\Delta_{n,r}^{1/r}
 \ge\bigl(\kappa q_{n,s}(\eta)^\beta\bigr)^{1/r}
 =\kappa^{1/r}q_{n,s}(\eta)^{\beta/r}.
\]

Lemma~\ref{prop:baseline} gives $e_{n,r}\downarrow L_r$, hence
$\widetilde e_{n,r}\downarrow0$ and $\widehat e_{n,r}\downarrow0$.
The sequence $q_n=q_{n,s}(\eta)$ is nonincreasing. Given $\varepsilon>0$, a finite $\varepsilon$-net of the compact set $\supp\eta$ gives $q_n\le\varepsilon$ once $n$ exceeds the net cardinality. Thus $q_n\downarrow0$.

The two bounds \eqref{eq:weighted-main-errors} have the form
$z_n\ge a q_n^\gamma$ with
\[
 (z_n,a,\gamma)=
 \left(\widetilde e_{n,r},\frac\kappa{rM_0^{r-1}},\beta\right),
 \qquad
 \left(\widehat e_{n,r},\kappa^{1/r},\frac\beta r\right).
\]
The lower comparison in Lemma~\ref{lem:log-transfer}(i) therefore gives
\begin{align*}
 \ldimqtilde r{P;S}&\ge\beta^{-1}\Dlo_s(\eta),&
 \udimqtilde r{P;S}&\ge\beta^{-1}\Dup_s(\eta),\\
 \ldimqhat r{P;S}&\ge(r/\beta)\Dlo_s(\eta),&
 \udimqhat r{P;S}&\ge(r/\beta)\Dup_s(\eta).
\end{align*}

Under the additional ball estimate for $\eta$,
\eqref{eq:Ap-rate} of Theorem~\ref{thm:master} gives
\[
 \Delta_{n,r}\ge\Lambda n^{-\beta/d},\qquad
 \Lambda=\kappa\left(\frac{d}{d+\beta/p}\right)^pM^{-\beta/d}>0.
\]
The scalar comparison already proved and the $r$th root yield
\[
 \widetilde e_{n,r}\ge\frac{\Lambda}{rM_0^{r-1}}n^{-\beta/d},
 \qquad
 \widehat e_{n,r}\ge\Lambda^{1/r}n^{-\beta/(rd)}.
\]
Applying Lemma~\ref{lem:log-transfer}(iii) to these two bounds gives
\[
 \ldimqtilde r{P;S}\ge d/\beta,\qquad
 \ldimqhat r{P;S}\ge rd/\beta.
\]

\end{proof}

\begin{example}[Uniformly prox-regular constraints]\label{ex:prox}\normalfont
A unit vector $v$ is a proximal normal to $S$ at $y$ if
$y\in\pi_S(y+tv)$ for some $t>0$. Suppose $S$ is uniformly
$\rho$-prox-regular, $0<\rho<\infty$, in the sense that
\begin{equation}\label{eq:normal}
 \langle v,a-y\rangle\le\frac{|a-y|^2}{2\rho}
 \quad(a\in S)
\end{equation}
for every such $y,v$. This is the exterior-ball formulation: expanding
$|y+\rho v-a|^2\ge\rho^2$ gives \eqref{eq:normal}; see
\cite[\S2, (2.1), and Theorem~4.1(d)]{csw1995}.
For Borel $F\subset K\cap\{b<\rho\}$ set
\begin{equation}\label{eq:prox-measure}
 w_\rho=1-b/\rho,\qquad \nu_{\rho,F}=T_*(w_\rho P|_F).
\end{equation}
Then
\begin{equation}\label{eq:prox-result}
 \Delta_{n,2}\ge\Vcal_{n,2}(\nu_{\rho,F};C)
 \ge\Vcal_{n,2}(\nu_{\rho,F};S)
 \ge\Vcal_{n,2}(\nu_{\rho,F}).
\end{equation}
For closed convex $S$, the same statement holds for any Borel $F\subset K$
with $w_\infty=1$. If $P(F)>0$ and $b\ge b_0>0$ on $F$, then for
$r\ge1$ the same measures also give
\begin{equation}\label{eq:prox-r}
 \Delta_{n,r}\ge\frac{rb_0^{r-1}}{2H}
             \Vcal_{n,2}(\nu_{\rho,F};C)
 \ge\frac{rb_0^{r-1}}{2H}\Vcal_{n,2}(\nu_{\rho,F}).
\end{equation}
\begin{proof}
For $y\in\pi_S(x)$ with $0<b(x)<\rho$, the vector
$v=(x-y)/b(x)$ is a unit proximal normal. Expansion of the square and
\eqref{eq:normal} show, for every $a\in S$, that
\begin{equation}\label{eq:prox-pointwise}
 |x-a|^2-b(x)^2
 =|a-y|^2-2b(x)\langle v,a-y\rangle
 \ge(1-b(x)/\rho)|a-y|^2.
\end{equation}
If $z$ is another nearest point, setting $a=z$ makes the left side zero
and forces $z=y$. When $b(x)=0$, the only nearest point is $x$ and
\eqref{eq:prox-pointwise} is an equality. Thus the fixed selector is the
unique projection on $F$. Since $0<w_\rho\le1$, the weight is integrable;
Theorem~\ref{thm:master} with $r=\beta=2$ proves
\eqref{eq:prox-result}. Its measure has positive mass when $P(F)>0$:
the sets $\{x\in F:w_\rho(x)\ge1/j\}$ cover $F$.

If $S$ is closed and convex, two distinct nearest points contradict
minimality at their midpoint. For its unique projection $y$, the segment
$y+t(a-y)\in S$ gives
$0\le-2t\langle x-y,a-y\rangle+t^2|a-y|^2$.
Divide by $t>0$ and let $t\downarrow0$ to obtain
$\langle x-y,a-y\rangle\le0$. The square expansion now yields
\eqref{eq:prox-pointwise} with weight one, which proves the convex case.

Finally, for $x\in F$ and $a\in C$, positivity of $b_0$ ensures $H>0$,
and
\begin{align*}
 |x-a|^r-b(x)^r
 &\ge rb_0^{r-1}(|x-a|-b(x))\\
 &=rb_0^{r-1}\frac{|x-a|^2-b(x)^2}{|x-a|+b(x)}
 \ge\frac{rb_0^{r-1}}{2H}w_\rho(x)|a-T(x)|^2.
\end{align*}
Apply Theorem~\ref{thm:master} with $\beta=2$ and
\[
 w(x)=\frac{rb_0^{r-1}}{2H}w_\rho(x),\qquad
 T_*(wP|_F)=\frac{rb_0^{r-1}}{2H}\nu_{\rho,F}.
\]
Since $\Vcal_{n,2}(c\nu;C)=c\Vcal_{n,2}(\nu;C)$ for $c\ge0$,
this gives \eqref{eq:prox-r}. 
If $F\ne\varnothing$, set
$\varepsilon=\inf_{x\in F}w_\rho(x)\in[0,1]$. Then every
$\alpha\in\Acal_n(S)$ satisfies
\[
 \int\dist(y,\alpha)^2\,d\nu_{\rho,F}(y)
 =\int_Fw_\rho(x)\dist(T(x),\alpha)^2\,dP(x)
 \ge\varepsilon\int_F\dist(T(x),\alpha)^2\,dP(x).
\]
Taking infima over $\alpha\in\Acal_n(S)$ gives
$\Delta_{n,2}\ge\varepsilon\Vcal_{n,2}(T_*(P|_F);S)$.
\end{proof}
\end{example}

\begin{example}[Uniform growth, the RKB bound {\cite[Theorem~4.1]{riegler2023lossy}}, and regular dimensions]\label{ex:growth}\normalfont
Suppose $\mu(B(y,t))\le Mt^d$ for all $y\in\R^D,t>0$, where $d,M>0$.
Assume the growth condition
\begin{equation}\tag{L2}\label{eq:L}
 \begin{gathered}
 F\subset K\text{ is Borel},\quad m=P(F)>0,\quad b\ge b_0>0\text{ on }F,\\
 |x-a|-b(x)\ge c|a-T(x)|^\beta
 \qquad(x\in F,\ a\in C),\qquad c,\beta>0.
 \end{gathered}
\end{equation}
For $1\le r<\infty$ put
\begin{equation}\label{eq:Ar}
 \lambda_r=rb_0^{r-1}c,\qquad
 A_r=\lambda_r M^{-\beta/d}m^{1+\beta/d}\frac{d}{d+\beta}.
\end{equation}
Then
\begin{equation}\label{eq:L-rates}
 \Delta_{n,r}\ge A_rn^{-\beta/d},\qquad
 \widetilde e_{n,r}\ge\frac{A_r}{rH^{r-1}}n^{-\beta/d},\qquad
 \widehat e_{n,r}\ge A_r^{1/r}n^{-\beta/(rd)}.
\end{equation}
Thus
\[
 \ldimqtilde{r}{P;S}\ge d/\beta,\qquad
 \ldimqhat{r}{P;S}\ge rd/\beta.
\] If $\mu$ is Ahlfors $d$-regular,
these combine with the upper estimates to give
\begin{equation}\label{eq:dimension-range}
 \begin{aligned}
 \frac{\HD\mu}{\beta}
 &\le\ldimqtilde{r}{P;S}\le\udimqtilde{r}{P;S}\le d,\\
 \frac{r\HD\mu}{\beta}
 &\le\ldimqhat{r}{P;S}\le\udimqhat{r}{P;S}\le rd,
 \qquad \HD\mu=d=\dim_HE_\mu.
 \end{aligned}
\end{equation}
Here $\HD\mu=\inf\{\dim_HA:A\text{ Borel},\ \mu(A)=1\}$.
Under condition U1 one also has
$\dim_HT(K)=\dim_HE=d$.
The same conclusion holds under U2, since Proposition~\ref{prop:U2}
gives U2 $\Rightarrow$ U1.
\begin{proof}
For fixed $b_0,\beta>0$, set
\[
 \gamma_\beta(x)=\inf_{\substack{a\in C\\a\ne T(x)}}
 \frac{|x-a|-b(x)}{|a-T(x)|^\beta},
\]
with empty infimum $\infty$. We show that
\[
 P\{x\in K:b(x)\ge b_0,\ \gamma_\beta(x)>0\}>0
\]
implies \eqref{eq:L}. For $c>0$, let
\begin{equation}\label{eq:set_F_delta}
 F_c^{(\beta)}
 =\{x\in K:b(x)\ge b_0,\ |x-a|-b(x)\ge c|a-T(x)|^\beta
                         \text{ for every }a\in C\}.
\end{equation}
For fixed $x\in K$, the function
\[
 a\longmapsto |x-a|-b(x)-c|a-T(x)|^\beta
\]
is continuous on $C$. If $\{a_\ell:\ell\ge1\}$ is countable and
dense in $C$, continuity therefore gives
\[
 F_c^{(\beta)}
 =\{b\ge b_0\}\cap K\cap
   \bigcap_{\ell\ge1}
   \{x:|x-a_\ell|-b(x)-c|a_\ell-T(x)|^\beta\ge0\}.
\]
Each function of $x$ inside this intersection is Borel because $b$
is continuous and $T$ is Borel. Thus $F_c^{(\beta)}$ is Borel. Moreover
\[
 \bigcup_{j\ge1}F_{1/j}^{(\beta)}
 =\{x:b(x)\ge b_0,\ \gamma_\beta(x)>0\}.
\]
Since $1/(j+1)\le1/j$, every $x\in F_{1/j}^{(\beta)}$ satisfies
the defining inequalities for $F_{1/(j+1)}^{(\beta)}$. Hence
\[
 F_{1/j}^{(\beta)}\subset F_{1/(j+1)}^{(\beta)},\qquad
 \lim_{j\to\infty}P(F_{1/j}^{(\beta)})
 =P\{x\in K:b(x)\ge b_0,\ \gamma_\beta(x)>0\}.
\]
If $m_*:=P\{x\in K:b(x)\ge b_0,\ \gamma_\beta(x)>0\}>0$,
convergence gives an index $j_0$
such that $P(F_{1/j_0}^{(\beta)})>m_*/2>0$. Thus
$F=F_{1/j_0}^{(\beta)}$ and $c=1/j_0$ satisfy \eqref{eq:L},
and both are fixed before $n$ is chosen. On such $F$, another nearest point
$a\in\pi_S(x)\subset C$ would give
$0\ge c|a-T(x)|^\beta$, so the projection there is unique.

If $\mu(B(z,t))\le\Gamma t^d$ for $z\in E_\mu$ and $0<t<t_0$,
Lemma~\ref{lem:ball-normalization}, with $M_\mu=1$, gives the global
ball estimate with
\[
 M=\max\{2^d\Gamma,(2/t_0)^d\}.
\]

For $x\in F$ and $a\in C$, \eqref{eq:L} gives
\[
 \begin{aligned}
 |x-a|^r-b(x)^r
 &=r\int_{b(x)}^{|x-a|}u^{r-1}\,du\\
 &\ge rb_0^{r-1}(|x-a|-b(x))\\
 &\ge\lambda_r|a-T(x)|^\beta.
 \end{aligned}
\]
Thus \eqref{eq:weighted-main-growth} holds with $w=\lambda_r$.
Put $\eta=T_*(P|_F/m)$. Then
\[
 \nu_F=m\lambda_r\eta,\qquad [m\lambda_r]_{A_p(\eta)}=1
 \quad(p>1),
\]
and Theorem~\ref{thm:master} yields
\[
 \Delta_{n,r}\ge\Vcal_{n,\beta}(\nu_F)
 =m\lambda_r\Vcal_{n,\beta}(\eta).
\]

Lemma~\ref{lem:uniform-growth}, applied with $A=C$, gives
\[
 \Delta_{n,r}\ge
 \lambda_r M^{-\beta/d}m^{1+\beta/d}\frac{d}{d+\beta}
 n^{-\beta/d}=A_rn^{-\beta/d}.
\]
Its proof applies \cite[Theorem~4.1]{riegler2023lossy} with source law and
reference measure $P_F=P|_F/m$. Its remaining error and dimension
estimates give \eqref{eq:L-rates} and
\[
 \ldimqtilde r{P;S}\ge d/\beta,\qquad
 \ldimqhat r{P;S}\ge rd/\beta.
\]

If, in addition, $\mu$ is Ahlfors $d$-regular, choose
$c_1,C_1,t_0>0$ such that
\[
 c_1t^d\le\mu(B(y,t))\le C_1t^d
 \quad(y\in E_\mu,\ 0<t<t_0).
\]
The lower ball estimate and Proposition~\ref{prop:packing} give
$c_n(E_\mu;E_\mu)\le C_2n^{-1/d}$.
Choose a minimizing covering codebook $\alpha_n\subset E_\mu$ and write
$\varepsilon_n=\sup_{E_\mu}\dist(y,\alpha_n)\le C_2n^{-1/d}$.
Since $T(x)\in E_\mu$ for $P$-a.e. $x$,
\[
 \dist(x,\alpha_n)\le b(x)+\varepsilon_n,
 \qquad\widetilde e_{n,r}\le\varepsilon_n.
\]
Using $e_{n,r}\le H+1$,
\[
 \Delta_{n,r}=r\int_{L_r}^{e_{n,r}}u^{r-1}\,du
 \le r(H+1)^{r-1}\varepsilon_n,
 \qquad\widehat e_{n,r}\le(r(H+1)^{r-1}C_2)^{1/r}n^{-1/(rd)}.
\]

Lemma~\ref{lem:log-transfer}(iii), applied with exponents $1/d$ and $1/(rd)$, gives
\[
 \udimqtilde r{P;S}\le d,\qquad
 \udimqhat r{P;S}\le rd.
\]
Together with the lower bounds, these give \eqref{eq:dimension-range}.

The two ball estimates and Proposition~\ref{prop:metric-dimension}
give $\dim_HE_\mu=d$. For every Borel set $A$ with $\mu(A)=1$,
the restricted measure $\mu|_{A\cap E_\mu}$ has mass one and,
for $a\in A\cap E_\mu$,
\[
 (\mu|_{A\cap E_\mu})(B_{A\cap E_\mu}(a,t))
 \le\mu(B(a,t))\le C_1t^d.
\]
That measure has total mass one, so $\dim_HA\ge d$.
Taking the infimum over full-mass Borel sets gives $\HD\mu=d$.

Under U1, $\mu(B(y,t))\ge P\{x:\pi_S(x)\subset B(y,t)\}>0$ for every $y\in E$ and $t>0$, so $E_\mu=E$.
Let $T(K)\subset\bigcup_iU_i$, with $d_i=\diam U_i<t_0/2$, and choose $y_i\in U_i\cap T(K)$ whenever this intersection is nonempty. Replace such $U_i$ by $B_i=B(y_i,2d_i)$ if $d_i>0$ and by $B_i=\{y_i\}$ otherwise. Then
$B=\bigcup_iB_i$ is Borel and $T^{-1}(B)=K$, whence
\[
 1=\mu(B)\le\sum_i\mu(B_i)\le C_12^d\sum_i d_i^d.
\]
The singleton terms have zero mass, since $\mu(\{y_i\})\le\lim_{t\downarrow0}C_1t^d=0$.
Thus $\mathcal H^d(T(K))\ge(C_12^d)^{-1}$, while $T(K)\subset E$ gives $\dim_HT(K)\le d$.
Hence $\dim_HT(K)=\dim_HE=d$.

\end{proof}
\end{example}

\begin{example}[Rectifiable projections and weighted coefficients]\label{ex:zador}\normalfont
Assume the hypotheses of Theorem~\ref{thm:master} with $\beta=2$.
Let $d\in\{1,\ldots,D\}$, suppose $E_\mu$ is countably
$d$-rectifiable, and write
\[
 \Hcal^d=\frac{\omega_d}{2^d}\mathcal H^d,\quad
 \omega_d=\frac{\pi^{d/2}}{\Gamma(1+d/2)},\quad
 \sigma=\Hcal^d|_{E_\mu},\quad d\mu=f\,d\sigma,\quad\int f\,d\sigma=1.
\]
Here $\mathcal H^d$ uses sums of $d$th powers of diameters; $\Hcal^d$
is normalized to agree with Lebesgue measure on affine $d$-planes.
Countable rectifiability means coverage, up to a null set, by countably
many Lipschitz images of bounded subsets of $\R^d$.
Let $Q_{2,d}$ denote the ordinary Euclidean quadratic Zador constant,
\[
 Q_{2,d}=\lim_{n\to\infty}n^{2/d}
 \Vcal_{n,2}(\mathcal L^d|_{[0,1]^d});
\]
it equals $C_{2,d}^2$ in \cite{aydin2026}. There is an induced density
$\wh$ with $d\nu_F=\wh\,d\mu$. Put
$q=d/(d+2)$ and $W=\int\wh\,d\mu$. Then
\begin{equation}\label{eq:zador}
 \liminf_{n\to\infty}n^{2/d}\Delta_{n,r}
 \ge Q_{2,d}\left(\int_{E_\mu}(\wh f)^q\,d\sigma\right)^{1/q}.
\end{equation}
If $f\in L^\infty(\sigma)$ and $M_f=\lVert f\rVert_\infty$, this
implies the following explicit lower coefficients whenever the
indicated weight hypotheses hold:
\begin{equation}\label{eq:coefficients}
 \liminf_{n\to\infty}n^{2/d}\Delta_{n,r}\ \ge\
 \begin{cases}
 Q_{2,d}M_f^{-2/d}W/[\wh]_{A_p(\mu)},&\wh\in A_p(\mu),\\[2pt]
 Q_{2,d}M_f^{-2/d}W[\wh]_{RH_s(\mu)}^{-2s/[d(s-1)]},
 &\wh\in RH_s(\mu).
 \end{cases}
\end{equation}
Here $1<p,s<\infty$ and
$[v]_{RH_s(\mu)}=\sup_B(\fint_Bv^s\,d\mu)^{1/s}/(\fint_Bv\,d\mu)$,
with the same balls and a.e. positivity convention as for $A_p$.
If instead $f\in L^u(\sigma)$, $1<u<\infty$, and $\wh\in A_p(\mu)$,
then
\begin{equation}\label{eq:Lu}
 \liminf_{n\to\infty}n^{2/d}\Delta_{n,r}
 \ge Q_{2,d}2^{-p-2/d-2/[d(u-1)]}
 \frac{W}{[\wh]_{A_p(\mu)}}
 \lVert f\rVert_u^{-2u/[d(u-1)]}.
\end{equation}
For Example~\ref{ex:prox} at $r=2$, use $\nu_F=\nu_{\rho,F}$.
For its estimate \eqref{eq:prox-r}, put
\[
 c_r=\frac{rb_0^{r-1}}{2H},\qquad
 d\nu_{\rho,F}=\wh\,d\mu,
 \qquad W=\int\wh\,d\mu.
\]
The measure in Theorem~\ref{thm:master} is then $c_r\nu_{\rho,F}$.
Its density and mass are $c_r\wh$ and $c_rW$, with
\[
 [c_r\wh]_{A_p(\mu)}=[\wh]_{A_p(\mu)},\qquad
 [c_r\wh]_{RH_s(\mu)}=[\wh]_{RH_s(\mu)},
\]
and
\[
 \left(\int(c_r\wh f)^q\,d\sigma\right)^{1/q}
 =c_r\left(\int(\wh f)^q\,d\sigma\right)^{1/q}.
\]
Substituting these quantities in \eqref{eq:zador},
\eqref{eq:coefficients}, and \eqref{eq:Lu} gives their bounds
with the unscaled density $\wh$.
\begin{proof}
Extend $w$ by zero off $F$ and call the resulting function $h$.
If $\mu(A)=0$, then $P(T^{-1}(A))=0$, so
$\nu_F(A)=\int_{T^{-1}(A)}h\,dP=0$. The Radon--Nikodym theorem
gives a nonnegative Borel density $\wh\in L^1(\mu)$, characterized by
\begin{equation}\label{eq:conditional}
 \int\varphi\wh\,d\mu=\int\varphi(T(x))h(x)\,dP(x)
 \quad(\varphi\text{ bounded Borel}).
\end{equation}
Taking indicator functions verifies
$\wh\circ T=\mathbb E_P[h\mid\sigma(T)]$ a.e.
If $0\le h\le1$, as in Example~\ref{ex:prox}, then
$\nu_F\le\mu$ and $0\le\wh\le1$ a.e. 

We have $d\nu_F=\wh f\,d\sigma$ on the compact countably rectifiable
set $E_\mu$. It is a finite Borel measure with finite second moment.
By \cite[Theorem~1.7]{aydin2026},
\[
 \liminf_{n\to\infty}n^{1/d}\Vcal_{n,2}(\nu_F)^{1/2}
 \ge C_{2,d}\left(\int(\wh f)^q\,d\sigma\right)^{1/(2q)}.
\]
Since $Q_{2,d}=C_{2,d}^2$ and
\[
 \liminf_n n^{2/d}\Vcal_{n,2}(\nu_F)
 =\left(\liminf_n n^{1/d}\Vcal_{n,2}(\nu_F)^{1/2}\right)^2,
\]
we obtain
\[
 \liminf_{n\to\infty}n^{2/d}\Vcal_{n,2}(\nu_F)
 \ge Q_{2,d}\left(\int(\wh f)^q\,d\sigma\right)^{1/q}.
\]
Since $\Delta_{n,r}\ge\Vcal_{n,2}(\nu_F)$ by
Theorem~\ref{thm:master}, this proves \eqref{eq:zador}. When $W=0$,
$\nu_F=0$ and $\wh=0$ $\mu$-a.e., so the right side is zero and
the conclusion is $\liminf_n n^{2/d}\Delta_{n,r}\ge0$.

For \eqref{eq:coefficients}, suppose either
$\wh\in A_p(\mu)$ with $1<p<\infty$, or
$\wh\in RH_s(\mu)$ with $1<s<\infty$.
For \eqref{eq:Lu}, suppose $\wh\in A_p(\mu)$ with $1<p<\infty$.
In each case $\wh>0$ $\mu$-a.e. and $\wh\in L^1(\mu)$, so
$0<W=\int\wh\,d\mu<\infty$. The density substitution yields
\[
 \int(\wh f)^q\,d\sigma=\int\wh^qf^{q-1}\,d\mu.
\]
On $\{f=0\}$, a $\mu$-null set, the right integrand is defined as zero.
If $f\le M_f$, then $q-1<0$ gives
\begin{equation}\label{eq:density-factor}
 \left(\int(\wh f)^q\,d\sigma\right)^{1/q}
 \ge M_f^{-2/d}\left(\int\wh^q\,d\mu\right)^{1/q}.
\end{equation}
For $\wh\in A_p(\mu)$ put $I=\int\wh^{-1/(p-1)}d\mu$.
As in \eqref{eq:global-Ap}, $WI^{p-1}\le[\wh]_{A_p(\mu)}$.
For $\theta=1/[1+q(p-1)]$, H\"older gives
\[
 1=\int\wh^{q\theta}\wh^{-q\theta}\,d\mu
 \le\left(\int\wh^q\,d\mu\right)^\theta I^{1-\theta}.
\]
Since $(1-\theta)/(q\theta)=p-1$, it follows that
\begin{equation}\label{eq:moment-Ap}
 \left(\int\wh^q\,d\mu\right)^{1/q}
 \ge I^{-(p-1)}\ge W/[\wh]_{A_p(\mu)}.
\end{equation}
For $\wh\in RH_s(\mu)$, the full-support ball instead gives
$\lVert\wh\rVert_s\le[\wh]_{RH_s}W$.
Set $\vartheta=q(s-1)/(s-q)$. The numbers
$q/\vartheta$ and $s/(1-\vartheta)$ are H\"older conjugates greater
than one, and hence
\[
 W\le\lVert\wh\rVert_q^{\vartheta}
         \lVert\wh\rVert_s^{1-\vartheta}.
\]
Dividing and raising to $1/\vartheta$ gives
\[
 \lVert\wh\rVert_q\ge
 W[\wh]_{RH_s}^{-(1-\vartheta)/\vartheta}
 =W[\wh]_{RH_s}^{-2s/[d(s-1)]}.
\]
Together with \eqref{eq:zador} and \eqref{eq:density-factor}, these
prove \eqref{eq:coefficients}. If
$\wh\in A_p(\mu)\cap RH_s(\mu)$, both inequalities in
\eqref{eq:coefficients} hold for the same sequence
$n^{2/d}\Delta_{n,r}$. Therefore
\[
 \liminf_{n\to\infty}n^{2/d}\Delta_{n,r}
 \ge Q_{2,d}M_f^{-2/d}W
 \max\left\{
 [\wh]_{A_p(\mu)}^{-1},\,
 [\wh]_{RH_s(\mu)}^{-2s/[d(s-1)]}
 \right\}.
\]

For the $L^u$ variant set
$M_u=(2\lVert f\rVert_u^u)^{1/(u-1)}$ and
$E_0=\{f\le M_u\}\cap E_\mu$. Then
\[
 \mu(E_\mu\setminus E_0)
 =\int_{f>M_u}f\,d\sigma
 \le M_u^{1-u}\int f^u\,d\sigma=\tfrac12.
\]
The H\"older calculation for \eqref{eq:moment-Ap}, restricted to
$E_0$, yields
\[
 \left(\int_{E_0}\wh^q\,d\mu\right)^{1/q}
 \ge\mu(E_0)^{1/q+p-1}I^{-(p-1)}
 \ge2^{-p-2/d}\frac{W}{[\wh]_{A_p(\mu)}}.
\]
Since $q-1<0$, the definition of $E_0$ gives
\[
 \begin{aligned}
 \int(\wh f)^q\,d\sigma
 &\ge\int_{E_0}\wh^qf^{q-1}\,d\mu\\
 &\ge M_u^{q-1}\int_{E_0}\wh^q\,d\mu.
 \end{aligned}
\]
As $(q-1)/q=-2/d$, it follows that
\[
 \begin{aligned}
 \left(\int(\wh f)^q\,d\sigma\right)^{1/q}
 &\ge M_u^{-2/d}\left(\int_{E_0}\wh^q\,d\mu\right)^{1/q}\\
 &\ge 2^{-p-2/d-2/[d(u-1)]}
 \frac{W}{[\wh]_{A_p(\mu)}}
 \lVert f\rVert_u^{-2u/[d(u-1)]}.
 \end{aligned}
\]
Substitution in \eqref{eq:zador} gives \eqref{eq:Lu}.

\end{proof}
\end{example}

\begin{example}[Growth checked on a restricted set of centers]\label{ex:perturb}\normalfont
Let $1\le r<\infty$ and $G\subset C$. Suppose that for every $n$
there is an optimal codebook
$\alpha_n\in C_{n,r}(P;S)\cap\Acal_n(C)$ contained in $\overline G$,
where $C_{n,r}(P;S)$ denotes the family of minimizers. Assume the
upper ball estimate
\[
 \mu(B(y,t))\le Mt^d\qquad(y\in\R^D,\ t>0),\qquad M,d>0,
\]
and assume \eqref{eq:L} only for centers $a\in G$, with the same
fixed $F,b_0,c,\beta$. Then all bounds \eqref{eq:L-rates} hold with
the same $A_r$, and their lower dimension consequences are
\[
 \ldimqtilde{r}{P;S}\ge d/\beta,\qquad
 \ldimqhat{r}{P;S}\ge rd/\beta.
\]
For $G=C\setminus E$, the containment
$\alpha_n\subset\overline G$ follows from
\[
 \alpha_n\cap E\subset\overline{S\setminus E}.
\]
\begin{proof}
Fix $x\in F$ and $a\in\overline G$. Choose $a_j\in G$ with $a_j\to a$. Since $C$ is closed, $a\in C$. The assumed growth inequality gives
\[
 |x-a|-b(x)-c|a-T(x)|^\beta
 =\lim_{j\to\infty}\bigl(|x-a_j|-b(x)-c|a_j-T(x)|^\beta\bigr)\ge0.
\]
Thus the same $F,b_0,c,\beta$ satisfy growth on $\overline G$.
Apply Lemma~\ref{lem:uniform-growth} with $A=\overline G$ and the optimal codebooks $\alpha_n$ from the hypothesis. It gives
\[
 \Delta_{n,r}\ge A_rn^{-\beta/d},\quad
 \widetilde e_{n,r}\ge\frac{A_r}{rH^{r-1}}n^{-\beta/d},\quad
 \widehat e_{n,r}\ge A_r^{1/r}n^{-\beta/(rd)}.
\]
Lemma~\ref{lem:log-transfer}(iii) gives
\[
 \ldimqtilde r{P;S}\ge\frac d\beta,\qquad
 \ldimqhat r{P;S}\ge\frac{rd}\beta.
\]

Suppose now that $G=C\setminus E$ and
$\alpha_n\cap E\subset\overline{S\setminus E}$.
For $a\in\alpha_n\cap E$, Lemma~\ref{lem:projection-topology} gives $|a|\le2R+|a_0|$. Given $\varepsilon>0$, choose $\widetilde a\in S\setminus E$ with $|\widetilde a-a|<\min\{1,\varepsilon\}$. Then
\[
 |\widetilde a|<2R+|a_0|+1,\qquad\widetilde a\in C\setminus E.
\]
For $a\in\alpha_n\setminus E$, take $\widetilde a=a$. Hence every $a\in\alpha_n$ is in $\overline{C\setminus E}$.
\end{proof}
\end{example}

\section{Geometric Examples}\label{section:examples}
\subsection{Separated convex constraints}\label{sec:convex-growth}

For a nonempty closed convex set $S\subset\R^D$, the nearest-point
projection is single-valued and nonexpansive:
\begin{equation}\label{eq:nonexpansive}
 |T(x)-T(y)|\le|x-y|\qquad(x,y\in\R^D);
\end{equation}
see \cite[Corollary~12.20]{rockafellar2009variational}.

\begin{remark}For example, take the closed unit disk as $S$ and
$K=\{(t,0):2\le t\le3\}$. For every $(t,0)\in K$, the nearest point is
$(1,0)$, since $|(t,0)-a|\ge t-|a|\ge t-1$ for $a\in S$, with equality
at $(1,0)$. Thus $T$ is constant on $K$ despite the smooth strictly convex
boundary of $S$.
\end{remark}

\begin{proposition}\label{prop:convex-growth}
Suppose $S$ is closed and convex and $K\cap S=\varnothing$.
Then \eqref{cond:Gbeta} holds with $F=K$, $\beta=2$,
$b_0=\min_K b>0$, and $c=(2H)^{-1}$.
\end{proposition}
\begin{proof}
Compactness of $K$ and closedness of $S$ give $b_0>0$. For $x\in K$, put $y=T(x)$. The convex normal inequality
\cite[Theorem~6.9 and Proposition~6.17]{rockafellar2009variational} gives
$(x-y)\cdot(a-y)\le0$ for $a\in S$. Hence
\[
 |x-a|^2-b(x)^2=|a-y|^2-2(x-y)\cdot(a-y)\ge|a-y|^2.
\]
For $a\in C$, $0<|x-a|+b(x)\le2H$, and therefore
\[
 |x-a|-b(x)=\frac{|x-a|^2-b(x)^2}{|x-a|+b(x)}
 \ge\frac{|a-T(x)|^2}{2H}.
\]
This proves the condition with constants fixed on $K\times C$.
\end{proof}

The positive-mass formulation also includes sources that meet the
constraint. For example, let $S=[0,1]\times\{0\}$ and let $P$ assign mass
$1/2$ uniformly to each of $[0,1]\times\{0\}$ and
$[0,1]\times\{1\}$. The projected measure is uniform on $S$.
For $F=[0,1]\times\{1\}$, one has $P(F)=1/2$ and $b=1$ on $F$;
on $[0,1]\times\{0\}$ one has $b=0$. Hence $L_r^r=1/2$.
The preceding projection inequality gives \eqref{cond:Gbeta} on $F$
with $\beta=2$. Corollary~\ref{thm:lower_asymptotics} therefore applies. 

\subsection{Circles with zero and positive baselines}
The exact order-$2$ solution in Proposition~\ref{prop:circle} is due to
Rosenblatt and Roychowdhury~\cite{rosenblatt2023uniform}.

\begin{proposition}\label{prop:circle}\label{example:quantizers_out_pi_s_k}
Let $P$ be the uniform probability on the unit circle
$K=\{(\cos t,\sin t):0\le t<2\pi\}$ and let $S=\overline B(0,1)$.
Then $L_r=0$ for every $r>0$, and for each fixed finite $r>0$,
\begin{equation}\label{eq:circle-all-r}
 e_{n,r}(P;S)\asymp n^{-1},\qquad
 \ldimqtilde{r}{P;S}=\udimqtilde{r}{P;S}=\ldimqhat{r}{P;S}=\udimqhat{r}{P;S}=1.
\end{equation}
At order $2$, define
\[
 \rho_n=\frac n\pi\sin\frac\pi n,\qquad
 \alpha_n=\left\{\rho_n\left(\cos\frac{2\pi j}{n},
             \sin\frac{2\pi j}{n}\right):j=0,\ldots,n-1\right\}.
\]
This is an optimal codebook, with $\alpha_1=\{0\}$, and
\begin{equation}\label{eq:circle-exact}
 e_{n,2}(P;S)^2=1-\left(\frac n\pi\sin\frac\pi n\right)^2,
 \qquad
 \widehat e_{n,2}=\widetilde e_{n,2}
 =\sqrt{1-\left(\frac n\pi\sin\frac\pi n\right)^2}.
\end{equation}
\end{proposition}
\begin{proof}
Since $K\subset S$, $T(x)=x$, $\mu=P$, and $L_r=0$. For $z\in K$ and $0<t\le1$, the chord formula gives
\[
 P(B(z,t))=\frac2\pi\arcsin(t/2),\qquad
 t/\pi\le P(B(z,t))\le t/2.
\]
Thus $P$ is Ahlfors $1$-regular. Proposition~\ref{prop:regular-rate} and
$\Acal_n(K)\subset\Acal_n(S)\subset\Acal_n(\R^2)$ yield, for every $r>0$,
\[
 c_rn^{-1}\le q_{n,r}(P)\le e_{n,r}(P;S)\le q_{n,r}(P;K)\le C_rn^{-1}.
\]
Here $\widetilde e_{n,r}=\widehat e_{n,r}=e_{n,r}$; Lemma~\ref{lem:log-transfer}(iii) gives all four dimensions equal to one.

At order $2$, \cite[Theorem~2.2.1]{rosenblatt2023uniform} gives the unrestricted optimum
$q_{n,2}(P)^2=1-\rho_n^2$ and an optimal regular polygon of radius
$\rho_n=(n/\pi)\sin(\pi/n)$. Rotation invariance of $P$ permits the angles $2\pi j/n$ used here. For $n\ge2$, $0<\rho_n<1$, so $\alpha_n\subset S$ and
\[
 1-\rho_n^2=q_{n,2}(P)^2
 \le e_{n,2}(P;S)^2
 \le F_2(\alpha_n)^2=1-\rho_n^2.
\]
For $n=1$, $\rho_1=0$ and $\int|x-a|^2\,dP(x)=1+|a|^2$, with minimum at $a=0\in S$. Thus \eqref{eq:circle-exact} holds for every $n\ge1$.
Finally, $\sin z/z=1-z^2/6+O(z^4)$ at $z=0$ gives
\[
 e_{n,2}^2=\frac{\pi^2}{3n^2}+O(n^{-4}),\qquad
 e_{n,2}=\frac\pi{\sqrt3\,n}(1+O(n^{-2})).
\]
\end{proof}

\begin{example}[A smaller concentric disk]\label{example:uniform_unit_circle}
Let $P$ be uniform on the unit circle $K\subset\R^2$, and let
$S=\overline B(0,1/2)$. Then $T(x)=x/2$, $E=\partial S$, and
$\mu=T_*P$ is uniform on $E$. Here $b=1/2$ on $K$. For every fixed
$0<r<\infty$,
\begin{equation}\label{eq:concentric-rate}
\Delta_{n,r}\asymp n^{-2},\qquad
\widetilde e_{n,r}\asymp n^{-2},\qquad
\widehat e_{n,r}\asymp n^{-2/r}.
\end{equation}
Consequently,
\[
 \ldimqtilde r{P;S}=\udimqtilde r{P;S}=\frac12,\qquad
 \ldimqhat r{P;S}=\udimqhat r{P;S}=\frac r2.
\] At infinite order, $\widetilde e_{n,\infty}\asymp n^{-2}$.
\end{example}
\begin{proof}
The formula for $T$ follows by minimizing $|x-a|$ on $|a|\le1/2$.
The circle ball estimate in Proposition~\ref{prop:circle}, after scaling,
shows that $\mu$ is Ahlfors $1$-regular.
For $x\in K$ and $a\in S$, the convex projection inequality gives
\[
|x-a|^2-\frac14\ge|a-x/2|^2.
\]
Also $1/4\le|x-a|^2\le9/4$. Fix $r>0$, and set
\[
u_r=\min_{1/4\le u\le9/4}\frac r2u^{r/2-1}>0,
\qquad
v_r=\max_{1/4\le u\le9/4}\frac r2u^{r/2-1}<\infty.
\]
For $s\in[1/4,9/4]$ the integral of this derivative gives
\begin{equation}\label{eq:circle-power-conversion}
u_r(s-1/4)\le s^{r/2}-(1/4)^{r/2}\le v_r(s-1/4).
\end{equation}
Consequently every codebook $\alpha\subset S$ with at most $n$ points satisfies
\begin{align*}
F_r(\alpha)^r-2^{-r}
 &\ge u_r\int_K\min_{a\in\alpha}|a-T(x)|^2\,dP(x)\\
 &=u_r\int_E\rho(y,\alpha)^2\,d\mu(y)
 \ge u_r q_{n,2}(\mu)^2.
\end{align*}
Proposition~\ref{prop:regular-rate} bounds the last squared error below
by a positive constant times $n^{-2}$. This gives the lower estimate for
$\Delta_{n,r}$.

For the upper estimate, take $n$ equally spaced centers on $E$. Every
$x\in K$ has a center with angular separation $|\theta|\le\pi/n$, and
\[
|x-a|^2=\frac14+1-\cos\theta,
\qquad 0\le1-\cos\theta\le\frac{\theta^2}{2}
 \le\frac{\pi^2}{2n^2}.
\]
Use the upper half of \eqref{eq:circle-power-conversion}, integrate, and
use this admissible codebook to get
$\Delta_{n,r}\le v_r\pi^2/(2n^2)$.
By Lemma~\ref{lem:conversion},
\[
 \frac{\Delta_{n,r}}{M_r}\le\widetilde e_{n,r}
 \le\frac{\Delta_{n,r}}{m_r}.
\]
Since $\Delta_{n,r}\asymp n^{-2}$, we obtain
\[
 \widetilde e_{n,r}\asymp n^{-2},\qquad
 \widehat e_{n,r}=\Delta_{n,r}^{1/r}\asymp n^{-2/r}.
\]
Thus
\[
-\log\widetilde e_{n,r}=2\log n+O(1),\qquad
-\log\widehat e_{n,r}=\frac2r\log n+O(1),
\]
which proves the stated dimensions.

For infinite order, the same convex projection inequality gives, for
every admissible $\alpha$,
\[
F_\infty(\alpha)^2-\frac14
 \ge\sup_{y\in E}\rho(y,\alpha)^2\ge c_n(E)^2.
\]
Since the square is increasing, taking the infimum gives
$e_{n,\infty}^2-1/4\ge c_n(E)^2\gtrsim n^{-2}$.
The equally spaced boundary centers give the reverse bound
$e_{n,\infty}^2-1/4\le\pi^2/(2n^2)$.
Finally,
\[
e_{n,\infty}-\frac12
 =\frac{e_{n,\infty}^2-1/4}{e_{n,\infty}+1/2},
\qquad 1\le e_{n,\infty}+1/2\le2.\qedhere
\]

\end{proof}

\begin{figure}[htbp]
\centering
\begin{tikzpicture}[scale=1.65]
\begin{scope}[xshift=-2.0cm]
\fill[orange!18] (0,0) circle (1);
\draw[thick,blue!70!black] (0,0) circle (1);
\foreach \a in {0,72,144,216,288}{\fill ({0.93549*cos(\a)},{0.93549*sin(\a)}) circle (0.028);}
\node at (0,1.23) {$K=\partial S$};
\node at (0,-1.24) {$L_r=0,\quad e_{n,r}\asymp n^{-1}$};
\end{scope}
\begin{scope}[xshift=2.0cm]
\draw[thick,blue!70!black] (0,0) circle (1);
\filldraw[fill=orange!18,draw=orange!70!black] (0,0) circle (0.5);
\foreach \a in {0,72,144,216,288}{\fill ({0.5*cos(\a)},{0.5*sin(\a)}) circle (0.028);}
\draw[->] (0.707,0.707)--(0.354,0.354);
\node at (0,1.23) {$K=\partial B(0,1)$};
\node at (0,-1.24) {$L_r=1/2,\quad\widetilde e_{n,r}\asymp n^{-2}$};
\end{scope}
\end{tikzpicture}
\caption{The configurations of Proposition~\ref{prop:circle} (left)
and Example~\ref{example:uniform_unit_circle} (right).
On the left, the five centers are optimal at order $2$ for the unit disk.
On the right, the five boundary centers of the smaller concentric disk
are those used for the upper estimate.}\label{fig:circles}
\end{figure}
\FloatBarrier

\subsection{A lifted Cantor distribution}
Pandey and Roychowdhury~\cite[Section 5.4]{pandey2024constrained}
determine the optimal codebooks and errors at order $2$ when the uniform
Cantor distribution is constrained to the segment $[0,1]\times\{1\}$.
The example below uses the equivalent full-line constraint and derives
excess rates for every fixed finite order $r>0$ from ordinary quantization.
The equivalence follows by moving each horizontal coordinate outside
$[0,1]$ to its nearest endpoint, as in the proof below.
\begin{example}\label{example:cantor}
Let $\mathcal C$ be the middle-third Cantor set, and let $\nu$ be its
uniform Cantor probability. Put
\[
K=\mathcal C\times\{0\},\qquad P=\nu\times\delta_0,
\qquad S=\R\times\{1\},\qquad d=\frac{\log2}{\log3}.
\]
Then $T(t,0)=(t,1)$, $L_r=1$, and for every fixed $r>0$,
\begin{equation}\label{eq:cantor-rate}
\Delta_{n,r}\asymp n^{-2/d},\qquad
\widetilde e_{n,r}\asymp n^{-2/d},\qquad
\widehat e_{n,r}\asymp n^{-2/(rd)}.
\end{equation}
Consequently,
\[
 \ldimqtilde r{P;S}=\udimqtilde r{P;S}=\frac d2,\qquad
 \ldimqhat r{P;S}=\udimqhat r{P;S}=\frac{rd}2.
\]
\end{example}
\begin{proof}
First we verify the regularity used in the estimate. A level-$k$ Cantor
interval has length $3^{-k}$ and mass $2^{-k}=(3^{-k})^d$.
Distinct level-$k$ intervals are separated by gaps of length at least
$3^{-k}$. Given $z\in\mathcal C$ and $0<t<1$, choose $k\ge0$ with
$3^{-(k+1)}\le t<3^{-k}$. An interval of length $2t$ meets at most three
level-$k$ intervals. Thus
\[
\nu(B(z,t))\le3\,2^{-k}\le3^{1+d}t^d.
\]
The level-$(k+1)$ interval containing $z$ lies in the closed radius-$t$
ball about $z$. Endpoint masses are zero, since every atom has mass at
most $2^{-j}$ for every $j$. Therefore the open ball also has mass at least
\[
2^{-(k+1)}=(3^{-(k+1)})^d\ge3^{-d}t^d.
\]
This proves Ahlfors $d$-regularity. The projected measure is a translate
of $\nu$, so it has the same estimates.

A center $(a,1)$ with $a\notin[0,1]$ can be moved horizontally to the
nearest endpoint of $[0,1]$ without increasing any source distance.
Consequently the infimum may be taken over horizontal center sets
$A\subset[0,1]$, and
\[
F_r(A\times\{1\})^r-1
 =\int_{\mathcal C}\bigl((1+\rho(t,A)^2)^{r/2}-1\bigr)\,d\nu(t).
\]
On $1\le u\le2$, the derivative $(r/2)u^{r/2-1}$ has a positive minimum
$a_r$ and a finite maximum $b_r$. Since $0\le\rho(t,A)\le1$, integration
of the derivative gives
\[
a_r\rho(t,A)^2\le(1+\rho(t,A)^2)^{r/2}-1
 \le b_r\rho(t,A)^2.
\]
The infimum of $\int\rho(t,A)^2\,d\nu$ over $A\subset[0,1]$, $|A|\le n$,
is $q_{n,2}(\nu)^2$: the same endpoint replacement applies to unrestricted
centers. We obtain
\[
a_r q_{n,2}(\nu)^2\le\Delta_{n,r}\le b_r q_{n,2}(\nu)^2.
\]
By Proposition~\ref{prop:regular-rate}, there are constants
$0<c_2\le C_2<\infty$, independent of $n$, such that
$c_2n^{-1/d}\le q_{n,2}(\nu)\le C_2n^{-1/d}$. Consequently,
\begin{equation}\label{eq:cantor-power-explicit}
a_rc_2^2n^{-2/d}\le\Delta_{n,r}\le b_rC_2^2n^{-2/d}.
\end{equation}
Here $L_r=1$ and every codebook with horizontal coordinates in
$[0,1]$ has pointwise distance at most $\sqrt2$, so
$1\le e_{n,r}\le\sqrt2$. Define
\[
\kappa_r=r\min\{1,2^{(r-1)/2}\}>0,\qquad
\Lambda_r=r\max\{1,2^{(r-1)/2}\}<\infty.
\]
Integrating $ru^{r-1}$ over $[1,e_{n,r}]$ gives
\[
\kappa_r(e_{n,r}-1)
\le r\int_1^{e_{n,r}}u^{r-1}\,du
=\Delta_{n,r}\le\Lambda_r(e_{n,r}-1).
\]
Combine this with \eqref{eq:cantor-power-explicit}, and take the
$r$th root of that same estimate, to obtain
\begin{align}
\frac{a_rc_2^2}{\Lambda_r}n^{-2/d}
&\le\widetilde e_{n,r}=e_{n,r}-1
 \le\frac{b_rC_2^2}{\kappa_r}n^{-2/d},\label{eq:cantor-additive-explicit}\\
(a_rc_2^2)^{1/r}n^{-2/(rd)}
&\le\widehat e_{n,r}=\Delta_{n,r}^{1/r}
 \le(b_rC_2^2)^{1/r}n^{-2/(rd)}.\label{eq:cantor-root-explicit}
\end{align}
All constants are positive and independent of $n$ for fixed $r>0$.
By \eqref{eq:cantor-additive-explicit},
\eqref{eq:cantor-root-explicit}, and Lemma~\ref{lem:log-transfer}(iii),
\[
 \lim_{n\to\infty}\frac{\log n}{-\log\widetilde e_{n,r}}=\frac d2,
 \qquad
 \lim_{n\to\infty}\frac{\log n}{-\log\widehat e_{n,r}}=\frac{rd}2.
\]
For $r\ge1$, the lower estimate is also an application of
Corollary~\ref{thm:lower_asymptotics} with $\beta=2$.
\end{proof}

\begin{figure}[htbp]
\centering
\begin{tikzpicture}[x=9cm,y=1cm,font=\small]
\draw[<->,orange!80!black,thick] (-0.04,1.35)--(1.04,1.35);
\node[anchor=west] at (1.055,1.35) {$S$};
\node[anchor=west] at (1.055,0) {$K$};
\draw[blue!75!black,line width=2pt] (0.00000000,0)--(0.00137174,0);
\draw[green!40!black,line width=2pt] (0.00000000,1.35)--(0.00137174,1.35);
\draw[blue!75!black,line width=2pt] (0.00274348,0)--(0.00411523,0);
\draw[green!40!black,line width=2pt] (0.00274348,1.35)--(0.00411523,1.35);
\draw[blue!75!black,line width=2pt] (0.00823045,0)--(0.00960219,0);
\draw[green!40!black,line width=2pt] (0.00823045,1.35)--(0.00960219,1.35);
\draw[blue!75!black,line width=2pt] (0.01097394,0)--(0.01234568,0);
\draw[green!40!black,line width=2pt] (0.01097394,1.35)--(0.01234568,1.35);
\draw[blue!75!black,line width=2pt] (0.02469136,0)--(0.02606310,0);
\draw[green!40!black,line width=2pt] (0.02469136,1.35)--(0.02606310,1.35);
\draw[blue!75!black,line width=2pt] (0.02743484,0)--(0.02880658,0);
\draw[green!40!black,line width=2pt] (0.02743484,1.35)--(0.02880658,1.35);
\draw[blue!75!black,line width=2pt] (0.03292181,0)--(0.03429355,0);
\draw[green!40!black,line width=2pt] (0.03292181,1.35)--(0.03429355,1.35);
\draw[blue!75!black,line width=2pt] (0.03566529,0)--(0.03703704,0);
\draw[green!40!black,line width=2pt] (0.03566529,1.35)--(0.03703704,1.35);
\draw[blue!75!black,line width=2pt] (0.07407407,0)--(0.07544582,0);
\draw[green!40!black,line width=2pt] (0.07407407,1.35)--(0.07544582,1.35);
\draw[blue!75!black,line width=2pt] (0.07681756,0)--(0.07818930,0);
\draw[green!40!black,line width=2pt] (0.07681756,1.35)--(0.07818930,1.35);
\draw[blue!75!black,line width=2pt] (0.08230453,0)--(0.08367627,0);
\draw[green!40!black,line width=2pt] (0.08230453,1.35)--(0.08367627,1.35);
\draw[blue!75!black,line width=2pt] (0.08504801,0)--(0.08641975,0);
\draw[green!40!black,line width=2pt] (0.08504801,1.35)--(0.08641975,1.35);
\draw[blue!75!black,line width=2pt] (0.09876543,0)--(0.10013717,0);
\draw[green!40!black,line width=2pt] (0.09876543,1.35)--(0.10013717,1.35);
\draw[blue!75!black,line width=2pt] (0.10150892,0)--(0.10288066,0);
\draw[green!40!black,line width=2pt] (0.10150892,1.35)--(0.10288066,1.35);
\draw[blue!75!black,line width=2pt] (0.10699588,0)--(0.10836763,0);
\draw[green!40!black,line width=2pt] (0.10699588,1.35)--(0.10836763,1.35);
\draw[blue!75!black,line width=2pt] (0.10973937,0)--(0.11111111,0);
\draw[green!40!black,line width=2pt] (0.10973937,1.35)--(0.11111111,1.35);
\draw[blue!75!black,line width=2pt] (0.22222222,0)--(0.22359396,0);
\draw[green!40!black,line width=2pt] (0.22222222,1.35)--(0.22359396,1.35);
\draw[blue!75!black,line width=2pt] (0.22496571,0)--(0.22633745,0);
\draw[green!40!black,line width=2pt] (0.22496571,1.35)--(0.22633745,1.35);
\draw[blue!75!black,line width=2pt] (0.23045267,0)--(0.23182442,0);
\draw[green!40!black,line width=2pt] (0.23045267,1.35)--(0.23182442,1.35);
\draw[blue!75!black,line width=2pt] (0.23319616,0)--(0.23456790,0);
\draw[green!40!black,line width=2pt] (0.23319616,1.35)--(0.23456790,1.35);
\draw[blue!75!black,line width=2pt] (0.24691358,0)--(0.24828532,0);
\draw[green!40!black,line width=2pt] (0.24691358,1.35)--(0.24828532,1.35);
\draw[blue!75!black,line width=2pt] (0.24965706,0)--(0.25102881,0);
\draw[green!40!black,line width=2pt] (0.24965706,1.35)--(0.25102881,1.35);
\draw[blue!75!black,line width=2pt] (0.25514403,0)--(0.25651578,0);
\draw[green!40!black,line width=2pt] (0.25514403,1.35)--(0.25651578,1.35);
\draw[blue!75!black,line width=2pt] (0.25788752,0)--(0.25925926,0);
\draw[green!40!black,line width=2pt] (0.25788752,1.35)--(0.25925926,1.35);
\draw[blue!75!black,line width=2pt] (0.29629630,0)--(0.29766804,0);
\draw[green!40!black,line width=2pt] (0.29629630,1.35)--(0.29766804,1.35);
\draw[blue!75!black,line width=2pt] (0.29903978,0)--(0.30041152,0);
\draw[green!40!black,line width=2pt] (0.29903978,1.35)--(0.30041152,1.35);
\draw[blue!75!black,line width=2pt] (0.30452675,0)--(0.30589849,0);
\draw[green!40!black,line width=2pt] (0.30452675,1.35)--(0.30589849,1.35);
\draw[blue!75!black,line width=2pt] (0.30727023,0)--(0.30864198,0);
\draw[green!40!black,line width=2pt] (0.30727023,1.35)--(0.30864198,1.35);
\draw[blue!75!black,line width=2pt] (0.32098765,0)--(0.32235940,0);
\draw[green!40!black,line width=2pt] (0.32098765,1.35)--(0.32235940,1.35);
\draw[blue!75!black,line width=2pt] (0.32373114,0)--(0.32510288,0);
\draw[green!40!black,line width=2pt] (0.32373114,1.35)--(0.32510288,1.35);
\draw[blue!75!black,line width=2pt] (0.32921811,0)--(0.33058985,0);
\draw[green!40!black,line width=2pt] (0.32921811,1.35)--(0.33058985,1.35);
\draw[blue!75!black,line width=2pt] (0.33196159,0)--(0.33333333,0);
\draw[green!40!black,line width=2pt] (0.33196159,1.35)--(0.33333333,1.35);
\draw[blue!75!black,line width=2pt] (0.66666667,0)--(0.66803841,0);
\draw[green!40!black,line width=2pt] (0.66666667,1.35)--(0.66803841,1.35);
\draw[blue!75!black,line width=2pt] (0.66941015,0)--(0.67078189,0);
\draw[green!40!black,line width=2pt] (0.66941015,1.35)--(0.67078189,1.35);
\draw[blue!75!black,line width=2pt] (0.67489712,0)--(0.67626886,0);
\draw[green!40!black,line width=2pt] (0.67489712,1.35)--(0.67626886,1.35);
\draw[blue!75!black,line width=2pt] (0.67764060,0)--(0.67901235,0);
\draw[green!40!black,line width=2pt] (0.67764060,1.35)--(0.67901235,1.35);
\draw[blue!75!black,line width=2pt] (0.69135802,0)--(0.69272977,0);
\draw[green!40!black,line width=2pt] (0.69135802,1.35)--(0.69272977,1.35);
\draw[blue!75!black,line width=2pt] (0.69410151,0)--(0.69547325,0);
\draw[green!40!black,line width=2pt] (0.69410151,1.35)--(0.69547325,1.35);
\draw[blue!75!black,line width=2pt] (0.69958848,0)--(0.70096022,0);
\draw[green!40!black,line width=2pt] (0.69958848,1.35)--(0.70096022,1.35);
\draw[blue!75!black,line width=2pt] (0.70233196,0)--(0.70370370,0);
\draw[green!40!black,line width=2pt] (0.70233196,1.35)--(0.70370370,1.35);
\draw[blue!75!black,line width=2pt] (0.74074074,0)--(0.74211248,0);
\draw[green!40!black,line width=2pt] (0.74074074,1.35)--(0.74211248,1.35);
\draw[blue!75!black,line width=2pt] (0.74348422,0)--(0.74485597,0);
\draw[green!40!black,line width=2pt] (0.74348422,1.35)--(0.74485597,1.35);
\draw[blue!75!black,line width=2pt] (0.74897119,0)--(0.75034294,0);
\draw[green!40!black,line width=2pt] (0.74897119,1.35)--(0.75034294,1.35);
\draw[blue!75!black,line width=2pt] (0.75171468,0)--(0.75308642,0);
\draw[green!40!black,line width=2pt] (0.75171468,1.35)--(0.75308642,1.35);
\draw[blue!75!black,line width=2pt] (0.76543210,0)--(0.76680384,0);
\draw[green!40!black,line width=2pt] (0.76543210,1.35)--(0.76680384,1.35);
\draw[blue!75!black,line width=2pt] (0.76817558,0)--(0.76954733,0);
\draw[green!40!black,line width=2pt] (0.76817558,1.35)--(0.76954733,1.35);
\draw[blue!75!black,line width=2pt] (0.77366255,0)--(0.77503429,0);
\draw[green!40!black,line width=2pt] (0.77366255,1.35)--(0.77503429,1.35);
\draw[blue!75!black,line width=2pt] (0.77640604,0)--(0.77777778,0);
\draw[green!40!black,line width=2pt] (0.77640604,1.35)--(0.77777778,1.35);
\draw[blue!75!black,line width=2pt] (0.88888889,0)--(0.89026063,0);
\draw[green!40!black,line width=2pt] (0.88888889,1.35)--(0.89026063,1.35);
\draw[blue!75!black,line width=2pt] (0.89163237,0)--(0.89300412,0);
\draw[green!40!black,line width=2pt] (0.89163237,1.35)--(0.89300412,1.35);
\draw[blue!75!black,line width=2pt] (0.89711934,0)--(0.89849108,0);
\draw[green!40!black,line width=2pt] (0.89711934,1.35)--(0.89849108,1.35);
\draw[blue!75!black,line width=2pt] (0.89986283,0)--(0.90123457,0);
\draw[green!40!black,line width=2pt] (0.89986283,1.35)--(0.90123457,1.35);
\draw[blue!75!black,line width=2pt] (0.91358025,0)--(0.91495199,0);
\draw[green!40!black,line width=2pt] (0.91358025,1.35)--(0.91495199,1.35);
\draw[blue!75!black,line width=2pt] (0.91632373,0)--(0.91769547,0);
\draw[green!40!black,line width=2pt] (0.91632373,1.35)--(0.91769547,1.35);
\draw[blue!75!black,line width=2pt] (0.92181070,0)--(0.92318244,0);
\draw[green!40!black,line width=2pt] (0.92181070,1.35)--(0.92318244,1.35);
\draw[blue!75!black,line width=2pt] (0.92455418,0)--(0.92592593,0);
\draw[green!40!black,line width=2pt] (0.92455418,1.35)--(0.92592593,1.35);
\draw[blue!75!black,line width=2pt] (0.96296296,0)--(0.96433471,0);
\draw[green!40!black,line width=2pt] (0.96296296,1.35)--(0.96433471,1.35);
\draw[blue!75!black,line width=2pt] (0.96570645,0)--(0.96707819,0);
\draw[green!40!black,line width=2pt] (0.96570645,1.35)--(0.96707819,1.35);
\draw[blue!75!black,line width=2pt] (0.97119342,0)--(0.97256516,0);
\draw[green!40!black,line width=2pt] (0.97119342,1.35)--(0.97256516,1.35);
\draw[blue!75!black,line width=2pt] (0.97393690,0)--(0.97530864,0);
\draw[green!40!black,line width=2pt] (0.97393690,1.35)--(0.97530864,1.35);
\draw[blue!75!black,line width=2pt] (0.98765432,0)--(0.98902606,0);
\draw[green!40!black,line width=2pt] (0.98765432,1.35)--(0.98902606,1.35);
\draw[blue!75!black,line width=2pt] (0.99039781,0)--(0.99176955,0);
\draw[green!40!black,line width=2pt] (0.99039781,1.35)--(0.99176955,1.35);
\draw[blue!75!black,line width=2pt] (0.99588477,0)--(0.99725652,0);
\draw[green!40!black,line width=2pt] (0.99588477,1.35)--(0.99725652,1.35);
\draw[blue!75!black,line width=2pt] (0.99862826,0)--(1.00000000,0);
\draw[green!40!black,line width=2pt] (0.99862826,1.35)--(1.00000000,1.35);
\foreach \t in {0,0.25,0.75,1}{
\draw[->,gray,dashed] (\t,0.07)--(\t,1.28);}
\node[fill=white,inner sep=2pt] at (0.5,0.68) {$T(t,0)=(t,1)$};
\node[anchor=south] at (0.5,1.5) {$P=\nu\times\delta_0,\quad L_r=1$};
\node[anchor=north] at (0,-0.06) {$0$};
\node[anchor=north] at (1,-0.06) {$1$};
\draw[blue!75!black,line width=2pt] (0.00000000,-0.850)--(1.00000000,-0.850);
\node[anchor=west] at (1.055,-0.850) {$\mathcal C_0$};
\draw[blue!75!black,line width=2pt] (0.00000000,-1.200)--(0.33333333,-1.200);
\draw[blue!75!black,line width=2pt] (0.66666667,-1.200)--(1.00000000,-1.200);
\node[anchor=west] at (1.055,-1.200) {$\mathcal C_1$};
\draw[blue!75!black,line width=2pt] (0.00000000,-1.550)--(0.11111111,-1.550);
\draw[blue!75!black,line width=2pt] (0.22222222,-1.550)--(0.33333333,-1.550);
\draw[blue!75!black,line width=2pt] (0.66666667,-1.550)--(0.77777778,-1.550);
\draw[blue!75!black,line width=2pt] (0.88888889,-1.550)--(1.00000000,-1.550);
\node[anchor=west] at (1.055,-1.550) {$\mathcal C_2$};
\draw[blue!75!black,line width=2pt] (0.00000000,-1.900)--(0.03703704,-1.900);
\draw[blue!75!black,line width=2pt] (0.07407407,-1.900)--(0.11111111,-1.900);
\draw[blue!75!black,line width=2pt] (0.22222222,-1.900)--(0.25925926,-1.900);
\draw[blue!75!black,line width=2pt] (0.29629630,-1.900)--(0.33333333,-1.900);
\draw[blue!75!black,line width=2pt] (0.66666667,-1.900)--(0.70370370,-1.900);
\draw[blue!75!black,line width=2pt] (0.74074074,-1.900)--(0.77777778,-1.900);
\draw[blue!75!black,line width=2pt] (0.88888889,-1.900)--(0.92592593,-1.900);
\draw[blue!75!black,line width=2pt] (0.96296296,-1.900)--(1.00000000,-1.900);
\node[anchor=west] at (1.055,-1.900) {$\mathcal C_3$};
\draw[blue!75!black,line width=2pt] (0.00000000,-2.250)--(0.01234568,-2.250);
\draw[blue!75!black,line width=2pt] (0.02469136,-2.250)--(0.03703704,-2.250);
\draw[blue!75!black,line width=2pt] (0.07407407,-2.250)--(0.08641975,-2.250);
\draw[blue!75!black,line width=2pt] (0.09876543,-2.250)--(0.11111111,-2.250);
\draw[blue!75!black,line width=2pt] (0.22222222,-2.250)--(0.23456790,-2.250);
\draw[blue!75!black,line width=2pt] (0.24691358,-2.250)--(0.25925926,-2.250);
\draw[blue!75!black,line width=2pt] (0.29629630,-2.250)--(0.30864198,-2.250);
\draw[blue!75!black,line width=2pt] (0.32098765,-2.250)--(0.33333333,-2.250);
\draw[blue!75!black,line width=2pt] (0.66666667,-2.250)--(0.67901235,-2.250);
\draw[blue!75!black,line width=2pt] (0.69135802,-2.250)--(0.70370370,-2.250);
\draw[blue!75!black,line width=2pt] (0.74074074,-2.250)--(0.75308642,-2.250);
\draw[blue!75!black,line width=2pt] (0.76543210,-2.250)--(0.77777778,-2.250);
\draw[blue!75!black,line width=2pt] (0.88888889,-2.250)--(0.90123457,-2.250);
\draw[blue!75!black,line width=2pt] (0.91358025,-2.250)--(0.92592593,-2.250);
\draw[blue!75!black,line width=2pt] (0.96296296,-2.250)--(0.97530864,-2.250);
\draw[blue!75!black,line width=2pt] (0.98765432,-2.250)--(1.00000000,-2.250);
\node[anchor=west] at (1.055,-2.250) {$\mathcal C_4$};
\draw[blue!75!black,line width=2pt] (0.00000000,-2.600)--(0.00411523,-2.600);
\draw[blue!75!black,line width=2pt] (0.00823045,-2.600)--(0.01234568,-2.600);
\draw[blue!75!black,line width=2pt] (0.02469136,-2.600)--(0.02880658,-2.600);
\draw[blue!75!black,line width=2pt] (0.03292181,-2.600)--(0.03703704,-2.600);
\draw[blue!75!black,line width=2pt] (0.07407407,-2.600)--(0.07818930,-2.600);
\draw[blue!75!black,line width=2pt] (0.08230453,-2.600)--(0.08641975,-2.600);
\draw[blue!75!black,line width=2pt] (0.09876543,-2.600)--(0.10288066,-2.600);
\draw[blue!75!black,line width=2pt] (0.10699588,-2.600)--(0.11111111,-2.600);
\draw[blue!75!black,line width=2pt] (0.22222222,-2.600)--(0.22633745,-2.600);
\draw[blue!75!black,line width=2pt] (0.23045267,-2.600)--(0.23456790,-2.600);
\draw[blue!75!black,line width=2pt] (0.24691358,-2.600)--(0.25102881,-2.600);
\draw[blue!75!black,line width=2pt] (0.25514403,-2.600)--(0.25925926,-2.600);
\draw[blue!75!black,line width=2pt] (0.29629630,-2.600)--(0.30041152,-2.600);
\draw[blue!75!black,line width=2pt] (0.30452675,-2.600)--(0.30864198,-2.600);
\draw[blue!75!black,line width=2pt] (0.32098765,-2.600)--(0.32510288,-2.600);
\draw[blue!75!black,line width=2pt] (0.32921811,-2.600)--(0.33333333,-2.600);
\draw[blue!75!black,line width=2pt] (0.66666667,-2.600)--(0.67078189,-2.600);
\draw[blue!75!black,line width=2pt] (0.67489712,-2.600)--(0.67901235,-2.600);
\draw[blue!75!black,line width=2pt] (0.69135802,-2.600)--(0.69547325,-2.600);
\draw[blue!75!black,line width=2pt] (0.69958848,-2.600)--(0.70370370,-2.600);
\draw[blue!75!black,line width=2pt] (0.74074074,-2.600)--(0.74485597,-2.600);
\draw[blue!75!black,line width=2pt] (0.74897119,-2.600)--(0.75308642,-2.600);
\draw[blue!75!black,line width=2pt] (0.76543210,-2.600)--(0.76954733,-2.600);
\draw[blue!75!black,line width=2pt] (0.77366255,-2.600)--(0.77777778,-2.600);
\draw[blue!75!black,line width=2pt] (0.88888889,-2.600)--(0.89300412,-2.600);
\draw[blue!75!black,line width=2pt] (0.89711934,-2.600)--(0.90123457,-2.600);
\draw[blue!75!black,line width=2pt] (0.91358025,-2.600)--(0.91769547,-2.600);
\draw[blue!75!black,line width=2pt] (0.92181070,-2.600)--(0.92592593,-2.600);
\draw[blue!75!black,line width=2pt] (0.96296296,-2.600)--(0.96707819,-2.600);
\draw[blue!75!black,line width=2pt] (0.97119342,-2.600)--(0.97530864,-2.600);
\draw[blue!75!black,line width=2pt] (0.98765432,-2.600)--(0.99176955,-2.600);
\draw[blue!75!black,line width=2pt] (0.99588477,-2.600)--(1.00000000,-2.600);
\node[anchor=west] at (1.055,-2.600) {$\mathcal C_5$};
\draw[blue!75!black,line width=2pt] (0.00000000,-2.950)--(0.00137174,-2.950);
\draw[blue!75!black,line width=2pt] (0.00274348,-2.950)--(0.00411523,-2.950);
\draw[blue!75!black,line width=2pt] (0.00823045,-2.950)--(0.00960219,-2.950);
\draw[blue!75!black,line width=2pt] (0.01097394,-2.950)--(0.01234568,-2.950);
\draw[blue!75!black,line width=2pt] (0.02469136,-2.950)--(0.02606310,-2.950);
\draw[blue!75!black,line width=2pt] (0.02743484,-2.950)--(0.02880658,-2.950);
\draw[blue!75!black,line width=2pt] (0.03292181,-2.950)--(0.03429355,-2.950);
\draw[blue!75!black,line width=2pt] (0.03566529,-2.950)--(0.03703704,-2.950);
\draw[blue!75!black,line width=2pt] (0.07407407,-2.950)--(0.07544582,-2.950);
\draw[blue!75!black,line width=2pt] (0.07681756,-2.950)--(0.07818930,-2.950);
\draw[blue!75!black,line width=2pt] (0.08230453,-2.950)--(0.08367627,-2.950);
\draw[blue!75!black,line width=2pt] (0.08504801,-2.950)--(0.08641975,-2.950);
\draw[blue!75!black,line width=2pt] (0.09876543,-2.950)--(0.10013717,-2.950);
\draw[blue!75!black,line width=2pt] (0.10150892,-2.950)--(0.10288066,-2.950);
\draw[blue!75!black,line width=2pt] (0.10699588,-2.950)--(0.10836763,-2.950);
\draw[blue!75!black,line width=2pt] (0.10973937,-2.950)--(0.11111111,-2.950);
\draw[blue!75!black,line width=2pt] (0.22222222,-2.950)--(0.22359396,-2.950);
\draw[blue!75!black,line width=2pt] (0.22496571,-2.950)--(0.22633745,-2.950);
\draw[blue!75!black,line width=2pt] (0.23045267,-2.950)--(0.23182442,-2.950);
\draw[blue!75!black,line width=2pt] (0.23319616,-2.950)--(0.23456790,-2.950);
\draw[blue!75!black,line width=2pt] (0.24691358,-2.950)--(0.24828532,-2.950);
\draw[blue!75!black,line width=2pt] (0.24965706,-2.950)--(0.25102881,-2.950);
\draw[blue!75!black,line width=2pt] (0.25514403,-2.950)--(0.25651578,-2.950);
\draw[blue!75!black,line width=2pt] (0.25788752,-2.950)--(0.25925926,-2.950);
\draw[blue!75!black,line width=2pt] (0.29629630,-2.950)--(0.29766804,-2.950);
\draw[blue!75!black,line width=2pt] (0.29903978,-2.950)--(0.30041152,-2.950);
\draw[blue!75!black,line width=2pt] (0.30452675,-2.950)--(0.30589849,-2.950);
\draw[blue!75!black,line width=2pt] (0.30727023,-2.950)--(0.30864198,-2.950);
\draw[blue!75!black,line width=2pt] (0.32098765,-2.950)--(0.32235940,-2.950);
\draw[blue!75!black,line width=2pt] (0.32373114,-2.950)--(0.32510288,-2.950);
\draw[blue!75!black,line width=2pt] (0.32921811,-2.950)--(0.33058985,-2.950);
\draw[blue!75!black,line width=2pt] (0.33196159,-2.950)--(0.33333333,-2.950);
\draw[blue!75!black,line width=2pt] (0.66666667,-2.950)--(0.66803841,-2.950);
\draw[blue!75!black,line width=2pt] (0.66941015,-2.950)--(0.67078189,-2.950);
\draw[blue!75!black,line width=2pt] (0.67489712,-2.950)--(0.67626886,-2.950);
\draw[blue!75!black,line width=2pt] (0.67764060,-2.950)--(0.67901235,-2.950);
\draw[blue!75!black,line width=2pt] (0.69135802,-2.950)--(0.69272977,-2.950);
\draw[blue!75!black,line width=2pt] (0.69410151,-2.950)--(0.69547325,-2.950);
\draw[blue!75!black,line width=2pt] (0.69958848,-2.950)--(0.70096022,-2.950);
\draw[blue!75!black,line width=2pt] (0.70233196,-2.950)--(0.70370370,-2.950);
\draw[blue!75!black,line width=2pt] (0.74074074,-2.950)--(0.74211248,-2.950);
\draw[blue!75!black,line width=2pt] (0.74348422,-2.950)--(0.74485597,-2.950);
\draw[blue!75!black,line width=2pt] (0.74897119,-2.950)--(0.75034294,-2.950);
\draw[blue!75!black,line width=2pt] (0.75171468,-2.950)--(0.75308642,-2.950);
\draw[blue!75!black,line width=2pt] (0.76543210,-2.950)--(0.76680384,-2.950);
\draw[blue!75!black,line width=2pt] (0.76817558,-2.950)--(0.76954733,-2.950);
\draw[blue!75!black,line width=2pt] (0.77366255,-2.950)--(0.77503429,-2.950);
\draw[blue!75!black,line width=2pt] (0.77640604,-2.950)--(0.77777778,-2.950);
\draw[blue!75!black,line width=2pt] (0.88888889,-2.950)--(0.89026063,-2.950);
\draw[blue!75!black,line width=2pt] (0.89163237,-2.950)--(0.89300412,-2.950);
\draw[blue!75!black,line width=2pt] (0.89711934,-2.950)--(0.89849108,-2.950);
\draw[blue!75!black,line width=2pt] (0.89986283,-2.950)--(0.90123457,-2.950);
\draw[blue!75!black,line width=2pt] (0.91358025,-2.950)--(0.91495199,-2.950);
\draw[blue!75!black,line width=2pt] (0.91632373,-2.950)--(0.91769547,-2.950);
\draw[blue!75!black,line width=2pt] (0.92181070,-2.950)--(0.92318244,-2.950);
\draw[blue!75!black,line width=2pt] (0.92455418,-2.950)--(0.92592593,-2.950);
\draw[blue!75!black,line width=2pt] (0.96296296,-2.950)--(0.96433471,-2.950);
\draw[blue!75!black,line width=2pt] (0.96570645,-2.950)--(0.96707819,-2.950);
\draw[blue!75!black,line width=2pt] (0.97119342,-2.950)--(0.97256516,-2.950);
\draw[blue!75!black,line width=2pt] (0.97393690,-2.950)--(0.97530864,-2.950);
\draw[blue!75!black,line width=2pt] (0.98765432,-2.950)--(0.98902606,-2.950);
\draw[blue!75!black,line width=2pt] (0.99039781,-2.950)--(0.99176955,-2.950);
\draw[blue!75!black,line width=2pt] (0.99588477,-2.950)--(0.99725652,-2.950);
\draw[blue!75!black,line width=2pt] (0.99862826,-2.950)--(1.00000000,-2.950);
\node[anchor=west] at (1.055,-2.950) {$\mathcal C_6$};
\end{tikzpicture}
\caption{Example~\ref{example:cantor}. The upper diagram shows the
source $\mathcal C\times\{0\}$ and its vertical projection onto
$S=\R\times\{1\}$; the green intervals indicate the projected set.
Both Cantor sets are drawn at the sixth finite stage. The lower rows
show the constructions $\mathcal C_k$, $0\le k\le6$, consisting of
$2^k$ intervals of length $3^{-k}$, each carrying mass $2^{-k}$.}
\label{fig:cantor-projection}
\end{figure}
\FloatBarrier

\subsection{Orders below one}

\begin{example}\label{ex:subunit}
For $0<r<1$, the following example attains
$\udimqtilde r{P;S}=d/r$ in \eqref{eq:regular-upper-dim2} with $d=1$.
Let $S=[0,1]\times\{0\}$, $z=(0,1)$, and
\[
 U_S(A)=\int_0^1\one_A(t,0)\,dt
 \quad(A\subset\R^2\text{ Borel}),\qquad
 P=\tfrac12 U_S+\tfrac12\delta_z.
\]
Then $E=S$ and $L_r^r=1/2$ for every $r>0$.

For the uniform law on a unit interval, the optimal $r$th-power error is
\[
 a_rn^{-r},\qquad a_r=\frac1{2^r(r+1)}.
\]
To check this for all $r>0$, order a codebook and let its interval Voronoi
cells have lengths $\ell_i$. On a fixed interval of length $\ell$, the
integral of $|t-a|^r$ is minimized at its midpoint and equals
$\ell^{r+1}/(2^r(r+1))$: when $a$ lies inside the interval, differentiation
of the integral gives $(a-u)^r-(v-a)^r$, and an outside center can first be
moved to an endpoint. Convexity of $\ell^{r+1}$ gives
$\sum_i\ell_i^{r+1}\ge n^{-r}$. Equal intervals with midpoint centers
attain the bound.

For every $\alpha\in\Acal_n(S)$,
\[
 \begin{aligned}
 F_r(\alpha)^r
 &=\frac12\int_S\rho(x,\alpha)^r\,dU_S(x)
   +\frac12\rho(z,\alpha)^r\\
 &\ge\frac{a_r}{2n^r}+\frac12.
 \end{aligned}
\]
For $n\ge2$, let $\beta_{n-1}$ be an optimal $(n-1)$-point codebook for
$U_S$ and set $\alpha_n=\beta_{n-1}\cup\{(0,0)\}$. Then
$\rho(z,\alpha_n)=1$ and
\[
 F_r(\alpha_n)^r-\frac12
 \le\frac12\int_S\rho(x,\beta_{n-1})^r\,dU_S(x)
 =\frac{a_r}{2(n-1)^r}.
\]
Taking infima gives
\[
 \frac{a_r}{2n^r}\le\Delta_{n,r}
 \le\frac{a_r}{2(n-1)^r}.
\]
Thus $\Delta_{n,r}\asymp n^{-r}$ and
$\widetilde e_{n,r}\asymp n^{-r}$ by \eqref{eq:conversion}. Therefore
\[
 \ldimqtilde r{P;S}=\udimqtilde r{P;S}=1/r.
\]
Also $\widehat e_{n,r}=\Delta_{n,r}^{1/r}\asymp n^{-1}$, so
\[
 \ldimqhat r{P;S}=\udimqhat r{P;S}=1.
\]
\end{example}

\begin{figure}[htbp]
\centering
\begin{tikzpicture}[x=8cm,y=2.3cm,font=\small]
\draw[blue!75!black,line width=2pt] (0,0)--(1,0);
\node[anchor=north] at (0.5,-0.12) {$S=[0,1]\times\{0\},\quad \tfrac12U_S$};
\fill[blue!75!black] (0,1) circle (2pt);
\node[anchor=west] at (0.035,1) {$z=(0,1),\quad P(\{z\})=\tfrac12$};
\draw[->,gray,dashed] (0,0.94)--(0,0.07);
\node[anchor=west] at (0.018,0.52) {$T(z)=(0,0)$};
\filldraw[fill=white,draw=black,thick] (0,0) circle (2.6pt);
\foreach \x in {0.125,0.375,0.625,0.875}{\fill (\x,0) circle (2pt);}
\node[anchor=north east] at (-0.01,0) {$0$};
\node[anchor=north west] at (1.01,0) {$1$};
\end{tikzpicture}
\caption{Example~\ref{ex:subunit}. Half of the source mass is uniform
on $S$ and half is concentrated at $z$. The illustrated codebook has
$n=5$: the open center at $(0,0)$ realizes the atom's distance baseline,
and the four solid centers are the midpoints of equal subintervals of
$S$. This is the codebook used for the upper estimate with $n-1=4$.}
\label{fig:subunit}
\end{figure}
\FloatBarrier

\subsection{The full projection set and the selected image}
\begin{example}[A projection point of zero selected mass]\label{example:cond_a_fails}
Let $P$ be uniform in arclength on
\[
K=\{(t,-|t|):-1\le t\le1\},
\]
and let
\[
S=\{(x,y):|y|=x-1\}\cup\{(x,y):|y|=-x-1\}
 \cup\{(x,y):y=1+|x|\}.
\]
The full projection set is
\[
E=\{(s,1-s):1\le s\le3/2\}
 \cup\{(-s,1-s):1\le s\le3/2\}\cup\{(0,1)\}.
\]
For $0<t\le1/2$, the unique nearest point to $(t,-t)$ is $(1,0)$.
For $1/2<t\le1$, it is $(t+1/2,1/2-t)$.
For negative $t$, the formulas follow by reflection in the vertical axis.
At $t=0$, the nearest-point set is $\{(-1,0),(1,0),(0,1)\}$.
\end{example}
\begin{proof}
For $t\ge0$, the orthogonal projection of $(t,-t)$ onto the line
$y=1-x$ has horizontal coordinate $t+1/2$. It lies on the ray $x\ge1$
exactly when $t\ge1/2$; otherwise the nearest point of that ray is its
endpoint $(1,0)$. The projection onto the other right-hand ray has
horizontal coordinate $1/2$, so the endpoint is again nearest there.
The closest points on the left-hand rays and the upper V are their
vertices $(-1,0)$ and $(0,1)$, whose squared distances are both
$(1+t)^2+t^2$. For $t>0$ these are larger than the squared distance
$(1-t)^2+t^2$ to $(1,0)$. This proves the formulas on the right half of
$K$; reflection proves the left half and direct calculation handles $t=0$.

Choose $T(0,0)=(0,1)$ and the unique nearest point elsewhere. Then
$T(K)=E$, but $E_\mu=E\setminus\{(0,1)\}$, since the source of the upper
vertex is a null singleton. For sufficiently small $\varepsilon>0$, no
nearest-point set is contained in $B((0,1),\varepsilon)$: at the origin
there are also the two horizontal vertices, and elsewhere the nearest
point lies on a lower segment. Hence the containment event in U1 is
empty for that ball. 
\end{proof}

\begin{figure}[htbp]
\centering
\begin{tikzpicture}[scale=1.65,font=\small]
\draw[gray,->] (1,0)--(2.1,1.1);
\draw[gray,->] (1,0)--(2.1,-1.1);
\draw[gray,->] (-1,0)--(-2.1,1.1);
\draw[gray,->] (-1,0)--(-2.1,-1.1);
\draw[gray,->] (0,1)--(1.7,2.7);
\draw[gray,->] (0,1)--(-1.7,2.7);
\draw[blue!75!black,line width=1.5pt] (-1,-1)--(0,0)--(1,-1);
\draw[orange!85!black,line width=2.2pt] (-1,0)--(-1.5,-0.5);
\draw[orange!85!black,line width=2.2pt] (1,0)--(1.5,-0.5);
\draw[gray,dashed] (0,0)--(0,1);
\draw[gray,dashed] (0,0)--(1,0);
\draw[gray,dashed] (0,0)--(-1,0);
\fill[blue!75!black] (0,0) circle (1.6pt);
\filldraw[fill=white,draw=orange!85!black,line width=1pt] (0,1) circle (2.6pt);
\fill[orange!85!black] (-1,0) circle (2pt);
\fill[orange!85!black] (1,0) circle (2pt);
\draw[->,green!40!black] (0.8,-0.8)--(1.3,-0.3);
\node[anchor=east] at (-0.08,0.85) {$(0,1)$};
\node[anchor=east] at (-0.08,-0.17) {$(0,0)$};
\node[anchor=north] at (0.45,-0.68) {$K$};
\node[anchor=west] at (1.48,-0.27) {$E_\mu$};
\node[anchor=west] at (1.66,1.14) {$S$};
\node[anchor=south] at (0,2.35) {$E=E_\mu\cup\{(0,1)\}$};
\end{tikzpicture}
\caption{Example~\ref{example:cond_a_fails}. The gray rays form $S$,
and the blue V is the source $K$. The two orange segments form
$E_\mu$. The open point $(0,1)$ also belongs to the full projection
set $E$, but its selected preimage is the zero-mass singleton
$\{(0,0)\}$. Dashed segments join the origin to its three nearest
points. The green arrow illustrates a unique projection for $t>1/2$.}
\label{fig:projection-support}
\end{figure}
\FloatBarrier

\begin{example}[A Dirac source]\label{example:dirac_measure}
Let $P=\delta_0$ on $\R^2$ and let $S$ be the unit circle. Then
$E=S$, while every selector sends $0$ to one point $y\in S$, and
$\mu=\delta_y$. Every nonempty codebook on $S$ has cost one at every
order. Thus $L_r=B=e_{n,r}=1$ and, for $0<r<\infty$,
\[
 \widetilde e_{n,r}=\Delta_{n,r}=\widehat e_{n,r}=0,
 \qquad \widetilde e_{n,\infty}=0.
\]
Under the zero convention,
\[
 \ldimqtilde r{P;S}=\udimqtilde r{P;S}
 =\ldimqhat r{P;S}=\udimqhat r{P;S}=0,
\]
and $\ldimqtilde\infty{P;S}=\udimqtilde\infty{P;S}=0$. A small ball in $S$ does not contain $\pi_S(0)=S$, so U1 fails.
\end{example}

\section*{Acknowledgement}
The author thanks the Departments of Mathematics at the University of
Pittsburgh and the University of Alabama at Birmingham for providing
supportive research environments. The author also thanks Mrinal Kanti
Roychowdhury for his help during the early stages of this work.

\appendix
\section{Dimension and Covering Facts}\label{app_supp_lemmas_proofs}
\subsection{Definitions}\label{app_supp_notation_definitions}
\label{section:intro_notations_definitions}
We use the unnormalized Hausdorff measure
\[
\mathcal H^s_\delta(A)=\inf\left\{\sum_i\diam(U_i)^s:
 A\subset\bigcup_iU_i,\ \diam(U_i)<\delta\right\},\qquad
\mathcal H^s(A)=\lim_{\delta\downarrow0}\mathcal H^s_\delta(A),
\]
and $\dim_H A=\inf\{s\ge0:\mathcal H^s(A)=0\}$.
The box dimensions use the radius-$t$ covering numbers of
Section~\ref{section:upper_quantization_dimension_estimate}.
Using sets of diameter $t$ in their definition gives the same dimensions,
since a radius-$t$ ball has diameter at most $2t$, and a set of diameter
$t$ lies in a radius-$t$ ball about any of its points.
For a compactly supported probability $\nu$, define
\[
\HD\nu=\inf_{\substack{A\text{ Borel}\\\nu(A)=1}}\dim_H A,
\quad
\Blo^*\nu=\inf_{\substack{A\subset\supp\nu\text{ Borel}\\\nu(A)=1}}\Blo A,
\quad
\Bup^*\nu=\inf_{\substack{A\subset\supp\nu\text{ Borel}\\\nu(A)=1}}\Bup A.
\]
A finite Borel measure $\eta$ with compact support $A$ is Ahlfors
$d$-regular if there are $c,C,t_0>0$ such that
\[
ct^d\le\eta(B(a,t))\le Ct^d\qquad(a\in A,\ 0<t<t_0).
\]
Either side alone is called a lower or upper Ahlfors estimate. A compact
set is Ahlfors $d$-regular if $\mathcal H^d|_A$ is a finite nonzero measure
with these estimates. Background on these dimensions is given in
\cite{falconer2013fractal,graf}.

\subsection{Conversion of excesses and logarithmic dimensions}\label{sec:dimensions}

\begin{lemma}\label{lem:conversion}
Fix $0<r<\infty$. If $L_r>0$, put
\[
 m_r=r\min\{L_r^{r-1},e_{1,r}^{r-1}\},\qquad
 M_r=r\max\{L_r^{r-1},e_{1,r}^{r-1}\}.
\]
These are positive finite constants, and for every $n$,
\begin{equation}\label{eq:conversion}
 m_r\widetilde e_{n,r}\le\Delta_{n,r}
 \le M_r\widetilde e_{n,r}.
\end{equation}
If $e_{n,r}>L_r$ for every $n\ge1$, then
\begin{equation}\label{eq:conversion-limit}
 \frac{\Delta_{n,r}}{\widetilde e_{n,r}}\longrightarrow rL_r^{r-1}.
\end{equation}
If $L_r=0$, then $K\subset S$, and
\begin{equation}\label{eq:zero-baseline}
 \widetilde e_{n,r}=\widehat e_{n,r}=e_{n,r}.
\end{equation}
\end{lemma}
\begin{proof}
Suppose $L_r>0$. Since $L_r\le e_{n,r}\le e_{1,r}$,
\[
 \Delta_{n,r}=r\int_{L_r}^{e_{n,r}}t^{r-1}\,dt.
\]
The derivative $rt^{r-1}$ lies between $m_r$ and $M_r$ throughout this fixed
interval. Integrating over $[L_r,e_{n,r}]$ gives \eqref{eq:conversion}, including
the case $e_{n,r}=L_r$. If the interval has positive length, the mean value
theorem gives a point $\xi_n\in(L_r,e_{n,r})$ with
$\Delta_{n,r}/\widetilde e_{n,r}=r\xi_n^{r-1}$. The baseline limit implies
$\xi_n\to L_r$, which proves \eqref{eq:conversion-limit}.

Suppose $L_r=0$. Then $b=0$ almost everywhere. If $b(x_0)>0$ for some
$x_0\in K$, continuity gives $\varepsilon>0$ such that
$b(x)>b(x_0)/2$ for $x\in B(x_0,\varepsilon)\cap K$, contradicting
\[
0=\int_K b^r\,dP
\ge\left(\frac{b(x_0)}2\right)^rP(B(x_0,\varepsilon))>0,
\]
where the last inequality uses $x_0\in K=\supp P$. Thus $b=0$ on $K$; closedness of $S$ gives $K\subset S$.
The equalities \eqref{eq:zero-baseline} now follow directly from the definitions.
\end{proof}

\begin{lemma}[Logarithmic transfer]\label{lem:log-transfer}
Let $a_n,b_n\ge0$ be nonincreasing sequences tending to zero. Set
\[
 \Dlo(a)=\liminf_n\frac{\log n}{-\log a_n},\qquad
 \Dup(a)=\limsup_n\frac{\log n}{-\log a_n}
\]
when $a_n>0$ for all $n$, and set both dimensions to zero when $a_n$ is eventually zero. Use the same convention for $b$. In the statements below, $c,C,\theta,\tau$ are positive constants independent of $n$, and the indices $n_j$ are positive integers.
\begin{enumerate}[(i)]
\item If $a_n\le Cb_n^\theta$ eventually, where $C,\theta>0$, then
\begin{equation}\label{eq:log-one-sided}
 \Dlo(a)\le\theta^{-1}\Dlo(b),\qquad \Dup(a)\le\theta^{-1}\Dup(b).
\end{equation}
If $a_n\ge cb_n^\theta$ eventually, then
\[
 \Dlo(a)\ge\theta^{-1}\Dlo(b),\qquad
 \Dup(a)\ge\theta^{-1}\Dup(b).
\]
\item If $cb_n^\theta\le a_n\le Cb_n^\theta$ eventually, then
\[
 \Dlo(a)=\theta^{-1}\Dlo(b),\qquad
 \Dup(a)=\theta^{-1}\Dup(b).
\]
\item If $a_n\le Cn^{-\tau}$ eventually, then $\Dup(a)\le1/\tau$.
If $a_n\ge cn^{-\tau}$ eventually, then $\Dlo(a)\ge1/\tau$.
If $cn^{-\tau}\le a_n\le Cn^{-\tau}$ eventually, then
\[
 \Dlo(a)=\Dup(a)=1/\tau.
\]
\item If $n_j\to\infty$ and $a_{n_j}\le Cn_j^{-\tau}$, then $\Dlo(a)\le1/\tau$.
\end{enumerate}
In particular, $L_r>0$ gives
\begin{equation}\label{eq:hat-tilde-dimension-proof}
 \ldimqhat r{P;S}=r\ldimqtilde r{P;S},\qquad
 \udimqhat r{P;S}=r\udimqtilde r{P;S}.
\end{equation}
When $L_r=0$,
\[
 \ldimqhat r{P;S}=\ldimqtilde r{P;S},\qquad
 \udimqhat r{P;S}=\udimqtilde r{P;S}.
\]
\end{lemma}
\begin{proof}
For the upper comparison in (i), an eventually zero $a$ has dimension zero, and an eventually zero $b$ forces $a$ to be eventually zero. Otherwise both sequences are positive. Put $A_n=-\log a_n$ and $B_n=-\log b_n$; both tend to infinity. For $0<\varepsilon<\theta$ and large $n$,
\[
 A_n\ge\theta B_n-\log C\ge(\theta-\varepsilon)B_n,
 \qquad\frac{\log n}{A_n}\le\frac1{\theta-\varepsilon}\frac{\log n}{B_n}.
\]
Take liminf and limsup separately and let $\varepsilon\downarrow0$.
For the lower comparison, rewrite it as
$b_n\le c^{-1/\theta}a_n^{1/\theta}$ and apply the proved upper comparison with the sequences interchanged. This gives
\[
 \Dlo(b)\le\theta\Dlo(a),\qquad\Dup(b)\le\theta\Dup(a).
\]
The two comparisons prove (ii). For $b_n=n^{-\tau}$, both dimensions are exactly $1/\tau$; (i) gives (iii).
For (iv), if $a$ is eventually zero the assertion follows from the convention. Otherwise, along the indicated indices,
\[
 \frac{\log n_j}{-\log a_{n_j}}
 \le\frac{\log n_j}{\tau\log n_j-\log C}\longrightarrow1/\tau.
\]
The liminf over all $n$ is at most the limsup along this subsequence, proving (iv).
Finally, \eqref{eq:conversion} gives
$\widehat e_{n,r}\asymp\widetilde e_{n,r}^{1/r}$ when $L_r>0$; apply (ii) with $\theta=1/r$. At $L_r=0$, \eqref{eq:zero-baseline} makes the sequences equal.
\end{proof}

The threshold formulation of these dimensions uses strict inequalities
in the exponent. Its precise form and proof are given in
Lemma~\ref{lem:thresholds}.

\subsection{Logarithmic thresholds}

\begin{lemma}\label{lem:thresholds}
Let $a_n$ be as above and let $s>0$. Its threshold relations are
\begin{align}
 \liminf_n n^{1/s}a_n&=\infty &&\text{if }s<\Dlo(a),&
 \liminf_n n^{1/s}a_n&=0 &&\text{if }s>\Dlo(a),\label{eq:lower-threshold}\\
 \limsup_n n^{1/s}a_n&=\infty &&\text{if }s<\Dup(a),&
 \limsup_n n^{1/s}a_n&=0 &&\text{if }s>\Dup(a).\label{eq:upper-threshold}
\end{align}

\end{lemma}
\begin{proof}
The eventually-zero case satisfies the assertions directly. Otherwise put
$d_n=\log n/(-\log a_n)$ for large $n$. If $s<\Dlo(a)$, choose a finite
$t$ with $s<t<\Dlo(a)$. Eventually $d_n>t$, so $a_n>n^{-1/t}$ and
$n^{1/s}a_n>n^{1/s-1/t}\to\infty$. If $s>\Dlo(a)$, choose
$\Dlo(a)<t<s$. Infinitely many $n$ satisfy $d_n<t$. Along that subsequence,
$n^{1/s}a_n<n^{1/s-1/t}\to0$, proving the lower-limit assertion.

If $s<\Dup(a)$, choose $s<t<\Dup(a)$. Infinitely many $n$ satisfy
$d_n>t$, and along those indices $n^{1/s}a_n>n^{1/s-1/t}\to\infty$.
Finally, if $s>\Dup(a)$, choose $\Dup(a)<t<s$. Eventually $d_n<t$, so
$n^{1/s}a_n<n^{1/s-1/t}\to0$. These choices also cover dimensions zero or
infinity whenever the corresponding comparison with the finite $s$ is possible.
\end{proof}

\subsection{Mass estimates and Hausdorff dimension}

\begin{proposition}\label{prop:metric-dimension}\label{lemma:connection_lower_upper_ahlfors_hausdorff}
Let $A$ be a nonempty separable metric space carrying a Borel measure
$\eta$ with $0<\eta(A)<\infty$, and fix $d>0$.
If $\eta(B(a,t))\ge ct^d$ for all $a\in A$ and $0<t<t_0$, then
$\mathcal H^d(A)<\infty$ and $\dim_H A\le d$.
If $\eta(B(a,t))\le Ct^d$ on that range, then
$\mathcal H^d(A)>0$ and $\dim_H A\ge d$.
\end{proposition}
\begin{proof}
For $0<t<t_0$, under the lower ball bound, a $t$-separated set has at most
$2^d\eta(A)/(ct^d)$ points, because its radius-$t/2$ balls are disjoint. A maximal such set is finite and its radius-$t$ balls cover $A$. Thus
\[
 \mathcal H^d_{3t}(A)\le\frac{2^d\eta(A)}{ct^d}(2t)^d
 =4^d\eta(A)/c.
\]
Letting $t\downarrow0$ gives $\mathcal H^d(A)<\infty$ and $\dim_HA\le d$.
Under the upper ball bound, $\eta(\{a\})=0$. For a cover $A\subset\bigcup_iU_i$ with $d_i=\diam U_i<t_0/3$, choose $a_i\in U_i$ whenever $U_i\ne\varnothing$.
If $d_i>0$, $U_i\subset B(a_i,2d_i)$; if $d_i=0$, its outer mass is zero. Hence
\[
 \eta(A)\le\sum_i\eta^*(U_i)\le C2^d\sum_i d_i^d,
 \qquad\mathcal H^d(A)\ge\eta(A)/(C2^d)>0.
\]
Thus $\dim_HA\ge d$.
\end{proof}

\subsection{Ordinary quantization of a regular measure}

For $r\ge1$, the ordinary rate below is also given in
\cite[Theorem~1.1 and Section~4]{kloeckner2012}.

\begin{proposition}\label{prop:regular-rate}
Suppose $\mu$ is a probability with compact support $F$ and
\begin{equation}\label{eq:ADR}
 c t^d\le\mu(B(y,t))\le C t^d
 \qquad(y\in F,\ 0<t<t_0),
\end{equation}
where $d,c,C,t_0>0$. For every fixed $r>0$, there are constants
$0<c_r\le C_r<\infty$, independent of $n$, such that
\begin{equation}\label{eq:regular-rate}
 c_r n^{-1/d}\le q_{n,r}(\mu)
 \le q_{n,r}(\mu;F)\le C_r n^{-1/d}.
\end{equation}
The same two-sided rate holds for $c_n(F)$ and $c_n(F;F)$.
\end{proposition}
\begin{proof}
Proposition~\ref{prop:packing} gives $c_n(F;F)\le C_Fn^{-1/d}$,
and $q_{n,r}(\mu;F)\le c_n(F;F)$ for every $r>0$.
Apply Lemma~\ref{lem:ball-normalization} with $\Gamma=C$, $M_\mu=1$,
$\zeta=\mu$, and $s=r$. With $C_*=\max\{2^dC,(2/t_0)^d\}$, it gives
\[
 q_{n,r}(\mu)\ge\left(\frac d{d+r}\right)^{1/r}
 C_*^{-1/d}n^{-1/d}.
\]
Finally,
\[
 q_{n,r}(\mu)\le q_{n,r}(\mu;F)\le c_n(F;F),\qquad
 q_{n,r}(\mu)\le c_n(F)\le c_n(F;F).
\]
These inequalities prove both quantization and covering rates.
\end{proof}

\subsection{Ball estimates and uniform growth}

\begin{lemma}[Ball normalization]\label{lem:ball-normalization}
Let $\mu$ be a finite Borel measure with nonempty compact support $E_\mu$ and mass $M_\mu>0$. Suppose
$\mu(B(z,t))\le \Gamma t^d$ for $z\in E_\mu$ and $0<t<t_0$, where $d,\Gamma,t_0>0$. Then
\begin{equation}\label{eq:all-balls}
 \mu(B(y,t))\le C_*t^d\quad(y\in\R^D,t>0),\qquad
 C_*=\max\{2^d\Gamma ,M_\mu(2/t_0)^d\}.
\end{equation}
If $0\le\zeta\le\mu$, $m'=\zeta(\R^D)$, $s>0$, and $\alpha\subset\R^D$ is nonempty with $|\alpha|\le n$, then
\[
 \int\dist(y,\alpha)^s\,d\zeta(y)
 \ge\frac d{d+s}C_*^{-s/d}(m')^{1+s/d}n^{-s/d}.
\]

\end{lemma}
\begin{proof}
If $0<t<t_0/2$ and $z\in B(y,t)\cap E_\mu$, then
$B(y,t)\cap E_\mu\subset B(z,2t)$; hence its mass is at most $2^d\Gamma t^d$. An empty intersection has mass zero. For $t\ge t_0/2$, use
$\mu(B(y,t))\le M_\mu\le M_\mu(2/t_0)^dt^d$.
For $m'>0$, the source law and reference measure $\zeta/m'$ have ball bound $(C_*/m')t^d$. Apply \cite[Theorem~4.1]{riegler2023lossy} with distortion $|y-a|^s$ and parameters
$(p,k,m,c,\delta_0)=(1,s,d,C_*/m',\infty)$, and multiply by $m'$. If $m'=0$, both sides vanish.
\end{proof}

\begin{lemma}[Uniform growth]\label{lem:uniform-growth}
Let $1\le r<\infty$, let $A$ be a nonempty Borel subset of $C$, and let $F\subset K$ be Borel. Suppose
\[
 m=P(F)>0,\quad b(x)\ge b_0>0\ (x\in F),\qquad
 \mu(B(y,t))\le Mt^d\quad(y\in\R^D,t>0),
\]
and
\[
 |x-a|-b(x)\ge c|a-T(x)|^\beta\quad(x\in F,a\in A),
 \qquad c,\beta,d,M>0.
\]
Put $\lambda_r=rb_0^{r-1}c$ and
\[
 \Lambda_{r,F}=\lambda_r M^{-\beta/d}m^{1+\beta/d}\frac d{d+\beta}.
\]
Then every $\alpha\in\Acal_n(A)$ satisfies
\[
 F_r(\alpha)^r-L_r^r\ge\Lambda_{r,F}n^{-\beta/d}.
\]
If, for each $n$, an optimal codebook for $e_{n,r}(P;S)$ belongs to $\Acal_n(A)$, then
\[
 \Delta_{n,r}\ge\Lambda_{r,F}n^{-\beta/d},\qquad
 \widetilde e_{n,r}\ge\frac{\Lambda_{r,F}}{rH^{r-1}}n^{-\beta/d},\qquad
 \widehat e_{n,r}\ge\Lambda_{r,F}^{1/r}n^{-\beta/(rd)}.
\]
Consequently $\ldimqtilde r{P;S}\ge d/\beta$ and
$\ldimqhat r{P;S}\ge rd/\beta$.

\end{lemma}
\begin{proof}
For $x\in F$ and $a\in A$, $|x-a|\ge b(x)\ge b_0$, so
\[
 \varrho_r(x,a):=|x-a|^r-b(x)^r
 =r\int_{b(x)}^{|x-a|}u^{r-1}\,du
 \ge\lambda_r|a-T(x)|^\beta.
\]
Use $(F,\mathcal B(F))$ as source space and $(A,\mathcal B(A))$ as reproduction space. Both the source law and the reference measure are
$P_F=P|_F/m$. The distortion $\varrho_r$ is nonnegative and Borel.
The parameters $(p,k,m,c,\delta_0)$ in \cite[Theorem~4.1]{riegler2023lossy} are
\[
 \left(1,\beta,d,c_{\mathrm{RKB}},\infty\right),\qquad
 c_{\mathrm{RKB}}=(M/m)\lambda_r^{-d/\beta}.
\]
The reference measure is finite and $dP_F/dP_F=1$, with $L^\infty(P_F)$ norm one. For $a\in A$ and $t>0$,
\begin{align*}
 P_F\{x:\varrho_r(x,a)^{1/\beta}<t\}
 &\le P_F\{x:|T(x)-a|<\lambda_r^{-1/\beta}t\}\\
 &\le m^{-1}\mu(B(a,\lambda_r^{-1/\beta}t))
 \le c_{\mathrm{RKB}}t^d.
\end{align*}
Thus the reference measure is $\varrho_r^{1/\beta}$-subregular of dimension $d$ at every scale. For $\alpha\in\Acal_n(A)$, the convention in
Section~\ref{section:basic_properties_quantization} gives a measurable
map $q_\alpha:F\to\alpha$ with
$\varrho_r(x,q_\alpha(x))=\min_{a\in\alpha}\varrho_r(x,a)$.
The cited theorem applied to this quantizer gives
\[
 \int_F\min_{a\in\alpha}\varrho_r(x,a)\,dP_F(x)
 \ge c_{\mathrm{RKB}}^{-\beta/d}\frac d{d+\beta}n^{-\beta/d}.
\]
For each such $\alpha$, the integrand $\dist(x,\alpha)^r-b(x)^r$ is nonnegative on $K$. Therefore
\begin{align*}
 F_r(\alpha)^r-L_r^r
 &\ge m\int_F\min_{a\in\alpha}\varrho_r(x,a)\,dP_F(x)\\
 &\ge m c_{\mathrm{RKB}}^{-\beta/d}\frac d{d+\beta}n^{-\beta/d}
 =\Lambda_{r,F}n^{-\beta/d}.
\end{align*}

For the optimal codebook assumed in the lemma,
$F_r(\alpha_n)=e_{n,r}$, so the first part gives the bound for $\Delta_{n,r}$.
Since $L_r^r\ge mb_0^r>0$ and $L_r\le e_{n,r}\le H$,
\[
 \Delta_{n,r}=r\int_{L_r}^{e_{n,r}}u^{r-1}\,du
 \le rH^{r-1}(e_{n,r}-L_r).
\]
Thus
\[
 \widetilde e_{n,r}\ge\frac{\Delta_{n,r}}{rH^{r-1}}
 \ge\frac{\Lambda_{r,F}}{rH^{r-1}}n^{-\beta/d},
 \qquad
 \widehat e_{n,r}=\Delta_{n,r}^{1/r}
 \ge\Lambda_{r,F}^{1/r}n^{-\beta/(rd)}.
\]
Lemma~\ref{prop:baseline} gives
$\widetilde e_{n,r}\downarrow0$ and $\widehat e_{n,r}\downarrow0$.
Applying Lemma~\ref{lem:log-transfer}(iii) with
$\tau=\beta/d$ and $\tau=\beta/(rd)$, respectively, yields
\[
 \ldimqtilde r{P;S}\ge d/\beta,\qquad
 \ldimqhat r{P;S}\ge rd/\beta.
\]
\end{proof}

\bibliographystyle{amsalpha-fullnames}
\bibliography{ref}
\end{document}